\documentclass[11pt]{amsart}
\usepackage{amsmath,amsfonts,amsthm,amscd,amssymb,graphicx}
\numberwithin{equation}{section}

\newtheorem{theorem}{Theorem}[section]

\newtheorem{lemma}[theorem]{Lemma}

\newtheorem{proposition}[theorem]{Proposition}

\newtheorem{corollary}[theorem]{Corollary}

\newtheorem{example}[theorem]{Example}
\newtheorem{counterexample}[theorem]{Counterexample}
\newtheorem{remark}[theorem]{Remark}
\newtheorem{definition}[theorem]{Definition}

\newcommand{\RR}{\mathbb{R}}
\newcommand{\cO}{\mathcal{O}}
\newcommand{\cS}{\mathcal{S}}
\newcommand{\CalM}{\mathcal{M}}
\newcommand{\CC}{\mathbb{C}}

\newcommand{\cE}{\mathcal{E}}

\newcommand{\CalU}{\mathcal{U}}

\newcommand{\DD}{{\mathbb D}}

\def\bb1{{1\!\!1}}

\def\CalU{\mathcal{U}}
\def\cG{\mathcal{G}}

\def\bU{{\bar{U}}}

\def\tx{\tilde x}

\def\R{\Re e}
\def\I{\Im m}

\def\cQ{\mathcal{Q}}
\def\cT{\mathcal{T}}

\def\cB{\mathcal{B}}
\def\cA{\mathcal{A}}
\def\cC{\mathcal{C}}

\newcommand{\wprod}[1]{\langle{#1}\rangle}

\newcommand{\txi}{{\tilde \xi}}

\begin{document}

\title[Stability of multi-dimensional viscous shocks]
{Stability of multi-dimensional viscous shocks for symmetric systems
with variable multiplicities}

\author[Toan Nguyen]{Toan Nguyen}

\date{Revised date: \today}

\thanks{
\indent I would like to thank Professor Kevin Zumbrun for suggesting
the problem and his many great advices, support, and helpful
discussions. I also thank the referees for their helpful comments
that greatly improved the exposition. This work was supported in
part by the National Science Foundation award number DMS-0300487.}

\address{Department of Mathematics, Indiana University, Bloomington, IN 47405}
\email{nguyentt@indiana.edu}

\maketitle

%\begin{abstract} We establish long-time stability
%of multi-dimensional viscous shocks of a general class of symmetric
%hyperbolic--parabolic systems with variable multiplicities, notably
%including the compressible magnetohydrodynamics (MHD) equations in
%dimensions $d\ge 2$. We show that the $L^2$ stability estimate for
%low-frequency regimes established by O. Gu\`es, G. M\'etivier, M.
%Williams, and K. Zumbrun (GMWZ) via the construction of degenerate
%Kreiss' symmetrizers, together with available high-frequency
%estimates for the solution operator investigated by K. Zumbrun, is
%sufficient for our analysis to provide the long-time stability of
%arbitrary-amplitude multi-dimensional viscous shocks with (possibly
%non-sharp) rates of decay, provided the uniform spectral, or Evans,
%stability condition. This extends the existing result of K. Zumbrun,
%by relaxing a constant multiplicity assumption  to a variable
%multiplicity assumption and dropping a technical assumption on
%structure of the so--called glancing set. The key idea to the
%improvement is to introduce a new simple argument for obtaining a
%$L^1\to L^p$ resolvent bound, replacing the one obtained by
%pointwise bounds on the Green kernel.
%
%\end{abstract}

\begin{abstract} We establish long-time stability of multi-dimensional viscous shocks
of a general class of symmetric hyperbolic--parabolic systems with
variable multiplicities, notably including the equations of
compressible magnetohydrodynamics (MHD) in dimensions $d\ge 2$. This
extends the existing result established by K. Zumbrun for systems
with characteristics of constant multiplicity to the ones with
variable multiplicity, yielding the first such a stability result
for (fast) MHD shocks. At the same time, we are able to drop a
technical assumption on structure of the so--called glancing set
that was necessarily used in previous analyses. The key idea to the
improvements is to introduce a new simple argument for obtaining a
$L^1\to L^p$ resolvent bound in low--frequency regimes by employing
the recent construction of degenerate Kreiss' symmetrizers by O.
Gu\`es, G. M\'etivier, M. Williams, and K. Zumbrun. Thus, at the
low-frequency resolvent bound level, our analysis gives an
alternative to the earlier pointwise Green's function approach of K.
Zumbrun. High--frequency solution operator bounds have been
previously established entirely by nonlinear energy estimates.

\end{abstract}

%\clearpage
\tableofcontents
%\clearpage
%%%%%%%%%%%%%%%%%
%\bigbreak

%\newcommand{\iprod}[1]{\langle{#1}\rangle}
%\newcommand{\wprod}[1]{\langle{#1}\rangle}

%%%%%%%%%%%%%%%%%%%%%%%%%%%%%%%%%%%%%%%%%%%%%%%%%%%%%%%%%%%%%%%%%%%%%%%%%%%%%%%%%%%%%%%%%%%%%%%%%%%%%%%%%%%%%%%
%\setcounter{section}{1}
%\setcounter{equation}{0}

\section{Introduction}
We consider a general system of viscous conservation laws %($d\ge 3$)% on the quarter-plane
($d\ge 2$)
\begin{equation}\label{sys}
\tilde U_t +  \sum_jF^j(\tilde U)_{x_j} = \sum_{jk}(B^{jk}(\tilde
U)\tilde U_{x_k})_{x_j}, \quad x\in \mathbb{R}^{d},\quad t>0,
\end{equation}
$\tilde U,F^j\in \mathbb{R}^n$, $B^{jk}\in\mathbb{R}^{n \times n}$,
$n\ge 2$, with initial data $\tilde U(x,0)=\tilde U_0(x)$, and a
planar viscous shock, connecting the endstates $U_\pm$:
\begin{equation}\label{profile}
\tilde U=\bU(x_1), \quad \lim_{x_1\to \pm\infty} \bU(x_1)=U_\pm.
\end{equation}

We study the long-time linearized and nonlinear stability of the
viscous shock $\bU$ under multi-dimensional perturbations of initial
data. The problem has been carefully and successfully investigated
by K. Zumbrun and his collaborators in \cite{Z2,Z3,Z4,GMWZ1}. There,
due to technical arguments of the analysis, the authors put
assumptions %of
on the %constant
multiplicity of hyperbolic characteristic roots and structure of the
so-called glancing set (see (H4)-(H5) below). The latter assumption
(which is automatically satisfied in dimensions $d = 1, 2$ and in
any dimension for rotationally invariant problems) assures the
glancing set to be confined to a finite union of smooth curves on
which the branching eigenvalue has constant multiplicity. This is
precisely to reduce the complexity of multi-variable matrix
perturbation problem when dealing with glancing blocks to a
simplified form of a two-variable perturbation problem. Whereas, the
constant multiplicity assumption %(H4)
excludes an important physical application, namely, the equations of
magnetohydrodynamics (MHD) in dimensions $d\ge2$. In the current
paper, we are able to relax the assumption of constant
multiplicities to variable multiplicities, allowing (fast) MHD
shocks to be treated and thus yielding for the first time the
long-time multi-dimensional stability for these shocks. In addition,
we are also able to drop the assumption on structure of the glancing
set at a price of having $t^{1/4}$ slower in decay rates in %cases of
dimensions $d\ge 3$.

Our main improvements rely on recent remarkable and technical works
of O. Gu\`es, G. M\'etivier, M. Williams, and K. Zumbrun
\cite{GMWZ5,GMWZ6} where the authors have obtained the $L^2$
stability estimates and small viscosity stability for the symmetric
systems with variable multiplicities via their construction of
Kreiss' symmetrizers. The idea is to employ these available
estimates to establish the long-time stability, or more precisely,
to derive a resolvent bound in low--frequency regimes. This will be
the main contribution of our present paper. High-frequency estimates
are already established by K. Zumbrun via elegant nonlinear energy
estimates for a very general class of symmetrizable systems,
including our class under consideration.

We would like to mention that the idea of using $L^2$ stability
estimates via the construction of degenerate Kreiss' symmetrizers to
attack the long-time stability problem has been investigated in
\cite{GMWZ1}. There the authors obtain the result under (H4)-(H5)
assumptions (and treat the strictly parabolic systems). In our
analysis, we avoid these technical assumptions, by introducing a
rather simpler argument %of
for $L^1\to L^p$ resolvent bounds in low-frequency regimes, which turns out to be the key %of
to the improvements. The analysis works precisely for the case of
dimensions $d\ge 3$. % where it is needed.
In dimension $d=2$ (the condition (H5) is now always satisfied), the
analysis of \cite{GMWZ1} indeed works even for the MHD shocks as we
are considering here by combining their later work in \cite{GMWZ6}
(though it was not stated there). In Section \ref{sec-H5}, we
represent a slightly modified version of \cite{GMWZ1} treating this
two--dimensional case, or more generally, cases with (H5) in a more
direct way. Once these low--frequency resolvent bounds are obtained,
the stability analysis follows in a standard fashion
\cite{Z2,Z3,Z4}. See Section \ref{discussions} for further
discussions.

%This has shed %the
%light on
%%attempting to obtain the long-time stability of shock profiles for
%the problem of obtaining long-time stability of shock profiles for
%these such systems, notably including the MHD equations. In the
%current paper, we address this issue, thus yielding for the first
%time the long-time stability of fast MHD shocks. These shocks were
%excluded in the previous works by the constant multiplicity
%assumption (H4).

%Roughly speaking, we show that the $L^2$ stability estimate for
%low-frequency regimes
%of \cite{GMWZ1,GMWZ6}, together with %known
%available high-frequency estimates, is sufficient for our analysis
%to obtain the long-time stability, with rates of decay, of the Lax
%or overcompressive shock in dimensions $d\ge 3$ and of the
%undercompressive shock in dimensions $d\ge 5$. In addition, in the
%two--dimensional case of Lax or overcompressive shocks, we give a
%considerably simple proof of the results, following a modified
%version of \cite{GMWZ1} (see Section \ref{sec-H5} below).
%

%%%%%%%%%%%%%%%%%%%%%%%%%%%%%%%%%%%%%%%%%%%%%%%%%%%%%%%%%%%%%%%%%%%%%%%%%%%%%%%%%%%%%%%%%%%%%%%%%%%%%%%%%%%%%%%
\subsection{Equations and assumptions}
We consider the general hyperbolic-parabolic system of conservation
laws \eqref{sys} in conserved variable $\tilde U$, with
$$\tilde U = \begin{pmatrix}\tilde u^I\\
\tilde u^{II}\end{pmatrix}, \quad B^{jk}=\begin{pmatrix}0 & 0 \\
b^{jk}_1 & b^{jk}_2\end{pmatrix},$$ $\tilde u^I\in \RR^{n-r}$,
$\tilde u^{II}\in \RR^{r}$, and
\begin{equation}\label{ellcond}\Re\sigma \sum_{jk} b_2^{jk}\xi_j\xi_k
\ge \theta |\xi|^2>0, \quad \forall \xi \in \RR^n\backslash
\{0\}.\end{equation}

Following \cite{Z3,Z4}, we assume that equations
\eqref{sys} can be written,
alternatively, after a triangular change of coordinates
\begin{equation}\label{Wcoord}
\tilde W:=\tilde W(\tilde U) =\begin{pmatrix}\tilde w^I(\tilde u^I)\\
 \tilde w^{II}(\tilde u^I, \tilde u^{II})\end{pmatrix},
\end{equation}
in {\em the quasilinear, partially symmetric hyperbolic-parabolic
form}
\begin{equation}\label{symmetric}\tilde A^0 \tilde W_t + \sum_j\tilde A^j\tilde W_{x_j} = \sum_{jk}(\tilde
B^{jk} \tilde W_{x_k})_{x_j} + \tilde G,
\end{equation}
where $$\tilde A^0=\begin{pmatrix}\tilde A^0_{11} & 0\\ 0&\tilde
A^0_{22}\end{pmatrix},\qquad\tilde A^1=\begin{pmatrix}\tilde
A^1_{11} & \tilde A^1_{12} \\ \tilde A^1_{21}&\tilde
A^1_{22}\end{pmatrix}, \qquad \tilde B^{jk}=\begin{pmatrix}0 & 0 \\
0 & \tilde b^{jk}\end{pmatrix}$$ and, defining $\tilde W_\pm:=\tilde
W(U_\pm)$,
\medskip

(A1) $\tilde A^j(\tilde W_\pm),\tilde A^0,\tilde A^1_{11}$ are
symmetric, $\tilde A^0\ge \theta_0>0$,
%NOTE: only \tilde A^1_{11} needed symmetric, unlike CIME
\medskip

(A2) for each $\xi \in \RR^d\setminus \{0\}$, no eigenvector of
$\sum_j\xi_j\tilde A^j(\tilde A^0)^{-1}(\tilde W_\pm)$ lies in the
kernel of $\sum_{jk}\xi_j\xi_k\tilde B^{jk}(\tilde A^0)^{-1}(\tilde
W_\pm)$,
\medskip

(A3) $\Re\sigma\sum\tilde b^{jk}\xi_j\xi_k\ge \theta|\xi|^2$, and
$\tilde G=\begin{pmatrix}0\\\tilde g
\end{pmatrix}$ with $\tilde g(\tilde W_x,\tilde W_x)=\cO(|\tilde
W_x|^2).$
\medskip

Along with the above structural assumptions, we make the following
technical hypotheses:
\medskip

(H0) $F^j, B^{jk}, \tilde A^0, \tilde A^j, \tilde B^{jk}, \tilde
W(\cdot), \tilde g(\cdot,\cdot) \in C^{s+1}$, for $s\ge [(d-1)/2]+2$
in our analysis of linearized stability, and $s\ge
s(d):=[(d-1)/2]+4$ in our analysis of nonlinear stability.
\medskip

(H1) The eigenvalues of $\tilde A^1_{11}$ are (i) distinct from the shock speed $s=0$; (ii) of common sign; and (iii) of constant multiplicity with respect to $U$.
\medskip

(H2) %The eigenvalues of $dF^1(U_\pm)$ are distinct and nonzero.
$\det (dF^1(U_\pm)) \ne 0.$
\medskip

(H3) Local to $\bU(\cdot)$, stationary solutions of \eqref{sys},
connecting $U_\pm$, form a smooth manifold $\{\bU^\delta(\cdot)\},
\delta \in \CalU \subset \RR^l$.
 \medbreak

(H4) The eigenvalues of $\sum_j \xi_jdF^j(U_\pm)$ have constant
multiplicity with respect to $\xi\in \RR^d$, $\xi\ne 0$.

\medbreak

Structural assumptions (A1)-(A3) and (H0)-(H2) are satisfied for gas
dynamics and MHD; see discussions in
\cite{MaZ4,Z3,Z4,GMWZ5,GMWZ6}.\\

\textbf{Alternative Hypothesis H4$'.$} The constant multiplicity
condition in Hypothesis (H4) holds for the compressible Navier–
Stokes equations whenever  is hyperbolic. However, the condition
fails always for the equations of viscous MHD. In the paper, we are
able to treat symmetric systems like the viscous MHD under the
following relaxed hypothesis. \medbreak

(H4') The eigenvalues of $\sum_j \xi_jdF^j(U_\pm)$ are either
semisimple and of constant multiplicity or totally nonglancing in
the sense of \cite{GMWZ6}, Definition 4.3. \medbreak

\begin{remark}\textup{There will be easily seen that our results also apply
to the case where the characteristic roots satisfy a (BS)
condition\footnote{Thanks to one of the referees for his pointing
out this extension.} (see Definition 4.9, \cite{GMWZ6}), a more
general situation than the constant multiplicity condition, ensuring
that a suitable generalized block structure condition is
satisfied. %, a more general situation than the constant multiplicity
%condition.
 See Remark \ref{rem-BS} for further discussion.}
\end{remark}

\begin{remark} \textup{Here we stress that we are able to
drop the following structural assumption, which is needed for the
analyses of \cite{Z2,Z3,Z4,GMWZ1}. }

\textup{(H5) The set of branch points of the eigenvalues of $(\tilde
A^1)^{-1}(i\tau \tilde A^0+ \sum_{j\ne 1} i\xi_j\tilde A^j)_\pm$,
$\tau \in \RR$, $\tilde \xi\in \RR^{d-1}$ is the (possibly
intersecting) union of finitely many smooth curves
$\tau=\eta_q^\pm(\tilde \xi)$, on which the branching eigenvalue has
constant multiplicity $s_q$ (by definition $\ge 2$).}
\end{remark}

\subsection{Shock profiles}\label{shockprofile} We recall the following classification
of shock profiles.\\

{\parindent=0pt {\bf Hyperbolic Classification.} Let $i_+$ denote
the dimension of the stable subspace of $dF^1(U_+)$, $i_-$ denote
the dimension of the unstable subspace of $dF^1(U_-)$, and $i := i_+
+i_-$. Indices $i_\pm$ count the number of incoming characteristics
from the right/left of the shock, while $i$ counts the total number
of incoming characteristics toward the shock. Then, the hyperbolic
classification of profile $\bU(\cdot)$, i.e., the classification of
the associated hyperbolic shock $(U_-,U_+)$, is
$$\left\{\begin{array}{lll} \mbox{Lax type} &\mbox{if} & i=n+1,\\
\mbox{Undercompressive} &\mbox{if} & i\le n,\\
\mbox{Overcompressive} &\mbox{if} & i\ge n+2. \end{array}\right.$$ }

In case all characteristics are incoming on one side, i.e. $i_+ = n$
or $i_- = n$, a shock is called {\em extreme}.\\

{\parindent=0pt {\bf Viscous Classification.} A complete description
of the viscous connection requires the further %over--compressibility
compressibility index $l$, where $l$ is defined as in (H3). In case
the connection is ``maximally'' transverse:
\begin{equation}\label{l-index} l = \left\{\begin{array}{lcr} 1\qquad&
\mbox{Lax or undercompressive case}\\ i-n & \mbox{overcompressive
case}\end{array}\right.\end{equation} we call the shock ``pure''
type, and classify it according to its hyperbolic type. Otherwise,
we call it ``mixed'' under/overcompressive type. Throughout this
paper, {\em we assume all viscous profiles are of pure, hyperbolic
type.}}\\

For further discussions, see \cite[Section 1.2]{Z2} or \cite[Section
1.2]{Z3}, and the references therein.

\subsection{The uniform Evans stability condition} The linearized equations of \eqref{sys} about $\bar U$ are
\begin{equation}\label{intro-lind}
 U_t = LU :=
\sum_{j,k}( B^{jk} U_{x_k})_{x_j} - \sum_{j}( A^{j} U)_{x_j}
\end{equation}
with initial data $U(0)=U_0$. Here, $B^{jk} := B^{jk}(\bU(x_1))$ and
$A^j U:= dF^j(\bU(x_1))U - [dB^{j1}(\bU(x_1))U]\bU_{x_1} (x_1).$

 A necessary condition for linearized
stability is {\em weak spectral stability}, defined as nonexistence
of unstable spectra $\Re \lambda
>0$ of the linearized operator $L$ about the wave. As described in
\cite{Z2,Z3}, this is equivalent to nonvanishing for all $\tilde
\xi\in \RR^{d-1}$, $\Re \lambda>0$ of the {\it Evans function}
$$
D_L(\tilde \xi, \lambda),
$$ (see equation \eqref{Evans1} in Appendix \ref{app-Evans}) a Wronskian associated with the Fourier-transformed eigenvalue
ODE. Let $\zeta = (\tilde \xi,\lambda)$. Introduce polar coordinates
$\zeta = \rho \hat \zeta$, with $\hat \zeta = (\hat{\tilde \xi},\hat
\lambda) \in S^d$. We also define $S^d_+ = S^d \bigcap \{\R\hat
\lambda \ge 0\}$.

\begin{definition}\label{strongspectral}
%\textup{
We define {\it strong spectral stability} as {\it uniform Evans
stability}:\medbreak

{\parindent=0pt (D) $D_L(\hat \zeta,\rho)$ vanishes to precisely
$l^{th}$ order at $\rho=0$ for all $\hat \zeta \in S^d_+$ and has no
other zeros in $S^d_+ \times \bar \RR_+$, where $l$ is the
%over-
compressibility index defined as in (H3) and \eqref{l-index}.}
\end{definition}

The spectral stability of arbitrary-amplitude shocks can be checked
efficiently by numerical Evans computations as in
\cite{HLyZ1,HLyZ2}.

\subsection{The GMWZ result}\label{L2stab-est} We recall the recent result of Gu\`es, M\'etivier, Williams, and
 Zumbrun for low-frequency regimes, and refer the reader to their
 original papers for the detail of statements and the proof.

\begin{theorem}[\cite{GMWZ6}, Theorems 3.7 and 3.9; \cite{GMWZ1}, Section 8]\label{theo-stab-est}Assume (A1)-(A3),
(H0)-(H3), and (H4').

Then, the strong spectral stability condition (D) implies the $L^2$
uniform stability estimate for low-frequency regimes (precisely
stated below, \eqref{max}, Section \ref{GMWZresults}).
\end{theorem}

\begin{example}[\cite{GMWZ6}, Section 8] \textup{Fast Lax' shocks for viscous MHD equations satisfy
the structural assumptions of Theorem
\ref{theo-stab-est}.}\end{example}

However, it is also shown that

\begin{counterexample}[\cite{GMWZ6}, Section 8]\label{counterex} \textup{Slow Lax' shocks for viscous MHD equations do not satisfy
the structural assumption (H4'), and thus Theorem
\ref{theo-stab-est} does not apply to these
cases.}\\\end{counterexample}

\subsection{Main results}
Our main results are as follows.

\begin{theorem}[Linearized stability]\label{theo-lin}
Assuming (A1)-(A3), (H0)-(H3), (H4'), and (D), we obtain the
asymptotic $L^1\cap H^{[(d-1)/2]+2} \rightarrow L^p$ stability of
\eqref{intro-lind} for all three types of shocks in dimensions $d\ge
3$, for any $2\le p\le \infty$, with rates of decay
\begin{equation}\begin{aligned} |U(t)|_{L^2}&\le C
(1+t)^{-\frac{d-2}{4}} |U_0|_{L^1\cap L^2},\\
|U(t)|_{L^p}&\le C (1+t)^{-\frac{d-1}{2}(1-1/p) + \frac1{4}}
|U_0|_{L^1\cap H^{[(d-1)/2]+2}},
\end{aligned}\end{equation}
provided that the initial perturbations $U_0$ are in $L^1\cap L^2$ for $p=2$, or in $L^1 \cap H^{[(d-1)/2]+2}$ for $p>2$.
\end{theorem}

\begin{theorem}[Nonlinear stability]\label{theo-nonlin}
Assuming (A1)-(A3), (H0)-(H3), (H4'), and (D), we obtain the
asymptotic $L^1\cap H^s \rightarrow L^p \cap H^s$ stability for %the
Lax or overcompressive %shock
shocks in dimension $d\ge 3$ and %the
undercompressive %shock
shocks in dimensions $d\ge 5$, for $s\ge s(d)$ as defined in (H0),
and any $2\le p\le \infty$, with rates of decay
\begin{equation}\begin{aligned} &|\tilde U(t)-\bU|_{L^p}\le C
(1+t)^{-\frac{d-1}{2}(1-1/p) + \frac1{4}} |U_0|_{L^1\cap
H^s}\\&|\tilde U(t)-\bU|_{H^s}\le C (1+t)^{-\frac{d-2}4}
|U_0|_{L^1\cap H^s},
\end{aligned}\end{equation}
provided that the initial perturbations $U_0:=\tilde U_0 - \bU$ are
sufficiently small in $L^1 \cap H^s$.
\end{theorem}

\begin{remark}\label{rate}
\textup{ The price of %relaxing Hypothesis (H4) and
dropping Hypothesis (H5) is that the obtained rate of decay is
degraded by $t^{1/4}$ as comparing to those established in
\cite{Z2,Z3,Z4} or Theorem \ref{theo-stabH5} below. Therefore the
rates are possibly not sharp. In fact, we believe that the sharp
rate of decay in $L^2$ is rather that of a $d$-dimensional heat
kernel and the sharp rate of decay in $L^\infty$ dependent on the
characteristic structure of the associated inviscid equations, as in
the constant-coefficient case \cite{HoZ1,HoZ2}.}
\end{remark}

Our next main result addresses the stability %in
for the two--dimensional case that is not covered by the above
theorems. We remark here that as shown in \cite{Z3}, page 321,
Hypothesis (H5) is automatically satisfied in dimensions $d = 1, 2$
and in any dimension for rotationally invariant problems. Thus, in
treating the two--dimensional case, we assume this hypothesis
without making any further restriction on structure of the systems.
Also since it turns out that the proof does not depend on the
dimensions, we state (and prove) the theorem in a general form as
follows, recovering previous results of K. Zumbrun (see
\cite[Theorem 5.5]{Z3}) for ``uniformly inviscid stable'' Lax or
over--compressive shocks with same decay rates.

\begin{theorem}[Two-dimensional case or cases with (H5)]\label{theo-stabH5}
Assume the same hypotheses as in Theorems \ref{theo-lin} and
\ref{theo-nonlin} with additional assumption (H5). Then Lax or
over--compressive shocks are asymptotically nonlinearly $L^1\cap H^s
\rightarrow L^p \cap H^s$ stable in dimensions $d\ge 2$, for any
$2\le p\le \infty$, with rates of decay
\begin{equation}\begin{aligned} &|\tilde U(t)-\bU|_{L^p}\le C
(1+t)^{-\frac {d-1}2(1-1/p)} |U_0|_{L^1\cap H^s}\\&|\tilde
U(t)-\bU|_{H^s}\le C (1+t)^{-\frac{d-1}4} |U_0|_{L^1\cap H^s},
\end{aligned}\end{equation}
provided that the initial perturbations $U_0:=\tilde U_0 - \bU$ are
sufficiently small in $L^1 \cap H^s$. Similar statement can be
stated for linearized stability with same decay rates.
\end{theorem}

%Theorem \ref{theo-stabH5} will be proved in Section \ref{sec-H5}. It
%will be seen that the assumption (H5) indeed makes our proof in low
%frequency regions considerably simple and independent of those of
%Theorems \ref{theo-lin}--\ref{theo-nonlin} which will be established
%in Section \ref{sec-lind}.

%\begin{remark} Under the additional assumption (H5), we obtain the same decay
%rates as
%
%\end{remark}

%\begin{remark}\label{rem-2d}\textup{As shown in \cite{Z3}, condition (H5) is automatically satisfied in
%dimensions $d = 1, 2$ and in any dimension for rotationally
%invariant problems. Thus, in term of removing this structural
%assumption, previous results of Zumbrun \cite{Z2,Z3} together with
%Theorems \ref{theo-lin} and \ref{theo-nonlin} complete the theory
%for the Lax or overcompressive case in all dimensions.}
%\end{remark}

\subsection{Discussion and open problems}\label{discussions} As
observed in \cite{Z3,Z4}, the high-frequency estimate on the
solution operator has already been established without the
structural assumptions (H4)-(H5), mainly relying on the damping
energy estimates. Hence we shall use it here as a black box. We %like
would like to draw the reader's attention to our recent work in
\cite{NZ2} for a great simplification of this original
high-frequency argument, requiring higher regularity of the forcing
$f$ (to credit, the simplification was based on an argument
introduced in \cite{KZ} for relaxation shocks).

The difficulty of relaxing Hypothesis (H4) and dropping (H5),
extending results in \cite{Z2,Z3,Z4} obtained by pointwise bound
approach, is that %in
there and in \cite{GMWZ1} the authors apply the diagonalization of
glancing blocks, where the hypotheses are required, to obtain rather
sharp bounds on resolvent kernel and resolvent solution. We rather
use the $L^2$ stability bound more directly, avoiding to get sharp
bounds on the adjoint problem where the diagonalization of glancing
blocks must be applied (see Section 12, \cite{GMWZ1}), and as a
consequence, avoiding the diagonalization error (denoted by $\beta$
in \cite{GMWZ1} or $\gamma_2$ in \cite{Z3}) at
the expense of slightly degraded decay, % in $L^2$, %, even though the
%decay in $L^\infty$ norm remains the same as
comparing to those reported in \cite{Z2,Z3,GMWZ1}. However, the loss
$t^{1/4}$ of decay is still sufficient to close our analysis for
dimensions $d\ge 3$ in the Lax or overcompressive case and for $d\ge
5$ in the undercompressive case. As already mentioned at the
beginning of the paper, this $L^1\to L^p$ resolvent bound will be the key %for
to the improvement. %The extension in the two--dimensional case turns
%out to be simple

Our analysis indeed applies to all applications covered by the GMWZ
small viscosity theory. Hence, the %only
remaining open problem is to
treat
%systems with variable multiplicities in dimension $d=2$ and
cases that are not covered by the GMWZ theory, that is, the cases
when the structural assumption (H4') of Theorem \ref{L2stab-est} is
not satisfied or more generally when the generalized block structure
fails. Counterexample \ref{counterex} is showing one of such
interesting but untreated cases, violating the structural assumption
(H4').

It is also worth mentioning that the undercompressive shock analysis
was carried out in \cite{Z3} only in nonphysical dimensions $d\ge
4$, and thus still remains open in dimensions for $d\le 3$ for
systems with or without assumptions (H4)-(H5). Finally, in our
forthcoming paper \cite{N2}, we have been able to carry out the
analysis for boundary layers in dimensions $d\ge 2$, extending our
recent results in \cite{NZ2} to systems with variable
multiplicities. It turns out that the analysis for the boundary
layer case is quite more delicate than those for the case of Lax or
overcompressive shocks that we are studying here.

%%%%%%%%%%%%%%%%%%%%%%%%%%%%%%%%%%%%%%%%%%%%%%%%%%%%%%%%%%%%%%%%%%%%%%%%%%%%%%%%%%%%%%%%%%%%%
\section{Linearized estimates}\label{sec-lind}
The linearized equations of \eqref{sys} about the profile $\bar U$
are
\begin{equation}\label{lind}
 U_t = LU :=
\sum_{j,k}( B^{jk} U_{x_k})_{x_j}
- \sum_{j}( A^{j} U)_{x_j}
\end{equation}
with initial data $U(0)=U_0$.

Then, we obtain the following proposition.
\begin{proposition}\label{prop-estS} Under the hypotheses of Theorems \ref{theo-lin} and \ref{theo-nonlin}, the solution operator $\cS(t):=e^{Lt}$ of the linearized equations may be decomposed into low frequency and high frequency parts (defined precisely below) as $\cS(t) = \cS_1(t) + \cS_2(t)$ satisfying
\begin{equation}\label{boundcS1}
\begin{aligned} |\cS_1(t) \partial_{x_1}^{\beta_1}\partial_{\tx}^{\tilde \beta} f|_{L^p_x} \le& C
(1+t)^{-\frac{d-1}{2}(1-1/p)+\frac 14- \frac{|\tilde
\beta|}2-(1-\alpha)\frac{\beta_1}{2}}|f|_{L^1_x}
\end{aligned}\end{equation} for all $2\le p\le \infty$, $d\ge 3$, and $\beta =
(\beta_1,\tilde\beta)$ with $\beta_1=0,1$ and $\alpha$ defined as
\begin{equation}\label{alpha}\alpha := \left\{\begin{array}{ll}0
&\mbox{  for Lax or overcompressive case},\\1 &\mbox{  for
undercompressive case,} \end{array}\right.
\end{equation} and
\begin{equation}\label{boundcS2}
\begin{aligned}|\partial_{x_1}^{\gamma_1}\partial_{\tx}^{\tilde\gamma}\cS_2(t)f|_{L^2} &\le C
e^{-\theta_1t}|f|_{H^{|\gamma_1|+|\tilde \gamma|}},\end{aligned}\end{equation} for $\gamma = (\gamma_1,\tilde\gamma)$ with $\gamma_1=0,1$.
\end{proposition}

Here, we use the same decomposition of solution operator $\cS(t)$ as
in the article of K. Zumbrun \cite{Z3}; see (5.152)--(5.153) in
\cite{Z3} or \eqref{cS1} below.

\subsection{High--frequency estimate} We observe that our relaxed
Hypothesis (H4') and the dropped Hypothesis (H5) only play a role in
low--frequency regimes. Thus, in course of obtaining the
high--frequency estimate \eqref{boundcS2}, we make here the same
assumptions as were made in \cite{Z3}, and therefore the same
estimate remains valid as claimed in \eqref{boundcS2} under our
current assumptions. We omit to repeat its proof here, and refer the
reader to the article \cite{Z3}, (5.16), Proposition 5.7, for the
original proof. See also a great simplification in \cite{NZ2},
Proposition 3.6 in treating the boundary layer case.

In the remaining of this section, we shall focus on proving the
bounds on low-frequency part $\cS_1(t)$ of the linearized solution
operator.

Taking the Fourier transform in $\tilde x:=(x_2,\dots,x_d)$ of
linearized equation \eqref{lind}, we obtain a family of eigenvalue
ODE
\begin{equation}\label{F-lind}\begin{aligned}
\lambda U=L_{\tilde \xi }U:=\overbrace{(B_{11}U')'-(A_1U)'}^{L_0U} -
&i\sum_{j\not=1}A_j\xi_jU + i\sum_{j\not=1}B_{j1}\xi_jU'
\\&+i\sum_{k\not=1}(B_{1k}\xi_kU)' -
\sum_{j,k\not=1}B_{jk}\xi_j\xi_k U.
\end{aligned}\end{equation}

\subsection{$L^2$ stability estimate for low frequencies}\label{GMWZresults} We
briefly recall the procedure (see \cite{GMWZ1}, page 75--85) of
reducing the eigenvalue equations to the block structure equations
and stating the $L^2$ estimate for low-frequency regimes by the
construction of degenerate symmetrizers.

Let $U=(u^I,u^{II})^T$ a solution of eigenvalue equations, that is,
$(L_\txi - \lambda)U =f$ where $L_\txi$ is defined as in
\eqref{F-lind}. Following \cite[Section 2.4]{Z3}, consider the
variable $W$ as usual
$$W:=\begin{pmatrix}u^{I}\\u^{II}\\z\end{pmatrix}$$ with $z:=b_1^{11}u^I_{x_1}
+ b^{11}_2u^{II}_{x_1}$. Then we can write equations of $W$ as a
first order system
\begin{equation}\label{1steqsW} \partial_{x_1} W =
\cG(x_1,\lambda,\txi)W+F,\end{equation}with
$F:=(A_*^{-1}f^I,0,f^{II})^{tr}$, where $f = (f^I,f^{II})^{tr}$ and
$A_*:=A^1_{11} - A^1_{12}(b_2^{11})^{-1}b_1^{11}$; thus, in
particular, $|F|\le C|f|$, for some constant $C$. It is not
necessary for us to carry out in detail the form of $\cG$; though,
see equation (2.65) of \cite{Z3}. Indeed, we are only interested in
the fact that bounds in $L^p$ of $W$ will give those of $U$ in the
same norm.

We go further as in \cite[page 75]{GMWZ1} to write this $(n+r)
\times(n+r)$ system on $\RR$ as an equivalent $2(n+r) \times 2(n+r)$
``doubled'' boundary problem on $x_1\ge 0$:
\begin{equation}\label{1steqstW} \begin{aligned}&\partial_{x_1} \tilde W = \tilde \cG(x_1,\lambda,\txi)\tilde W+\tilde F\\&\Gamma \tilde W =0 \mbox{  on  }x_1=0\end{aligned}\end{equation} where \begin{equation} \begin{aligned}& \tilde W (x_1,\lambda,\txi) = (W_+,W_-), \\ &
\tilde \cG(x_1,\lambda,\txi) = \begin{pmatrix}\cG_+ & 0\\0&-\cG_-\end{pmatrix},\\
&\tilde F = \begin{pmatrix} F_+ \\-F_-\end{pmatrix},\\& \Gamma
\tilde W = W_+ - W_-\end{aligned}\end{equation} with $F_\pm(x_1) :=
F(\pm x_1)$.

For small or bounded frequencies $(\lambda,\txi)$, we use the known
MZ conjugation; see, for example, \cite{MeZ1} or \cite[Lemma
5.1]{GMWZ1}. That is, given any
$(\underline{\lambda},\underline{\txi})\in \RR^{d+1}$, there is a
smooth invertible matrix $\Phi(x_1,\lambda,\txi)$ for $x_1\ge 0$ and
$(\lambda,\txi)$ in a small neighborhood of
$(\underline{\lambda},\underline{\txi})$, such that \eqref{1steqstW}
is equivalent to
\begin{equation}\label{1steqsY} \partial_{x_1} Y = \cG_+(\lambda,\txi)Y+\tilde{\tilde F}, \quad \tilde \Gamma(\lambda,\txi)Y= 0\end{equation}
where $\cG_+(\lambda,\txi) := \tilde \cG(+\infty,\lambda,\txi),
\tilde W = \Phi Y, \tilde{\tilde F}=\Phi^{-1}\tilde F $ and $\tilde
\Gamma Y:= \Gamma\Phi Y$.

Next, there are smooth matrices $V(\lambda,\txi)$ such that
\begin{equation} \label{eq-bl}V^{-1}\cG_+ V =
\begin{pmatrix}H&0\\0&P\end{pmatrix}\end{equation} with blocks
$H(\lambda,\txi)$ and
\begin{equation}\label{para-bl}P(\lambda,\txi)=\begin{pmatrix}P_+&0\\0&P_-\end{pmatrix}\end{equation}
satisfying the eigenvalues $\mu$ of $P_\pm$ in $\{\pm\R\mu\ge c>0\}$
and
$$\begin{aligned}H(\lambda,\txi) &= H_0(\lambda,\txi) +
\cO(\rho^2)\\H_0(\lambda,\txi):&=-(A_{+}^1)^{-1}\Big((i\tau +
\gamma)A^0_+ + \sum_{j=2}^di\xi_jA^j_+\Big).\end{aligned}$$

Define variables $Z=(u_H,u_P)^T$ with $u_P:=(u_{P_+},u_{P_-})^T$ as
$$\tilde W = \Phi Y = \Phi V Z,\qquad \bar \Gamma Z:= \Gamma \Phi V
Z,$$ and $ (F_H,F_P)^T = V^{-1}\tilde{\tilde F}$. We have
\begin{equation}\label{1steqsZ}\partial_{x_1}\begin{pmatrix}u_H\\u_{P_\pm}\end{pmatrix}
= \begin{pmatrix}H&0\\0&{P_\pm}\end{pmatrix}
\begin{pmatrix}u_H\\u_{P_\pm}\end{pmatrix} +
\begin{pmatrix}F_H\\F_{P_\pm}\end{pmatrix}, \quad \bar \Gamma Z =
0.\end{equation}

Let $\wprod{\cdot,\cdot}$ denote the standard $L^2$ product over
$[0,\infty)$, that is, $$\wprod{f,g} = \int_0^\infty f(x_1)\bar
g(x_1) dx_1, \qquad \forall ~f,g\in L^2(0,\infty),$$where $\bar g$
is the complex conjugate of $g$.

Then, recalling that $\rho=|(\tilde \xi,\lambda)|$ and $\gamma =
\R\lambda$, we obtain the maximal stability estimate for the low
frequency
regimes (\cite{GMWZ6,GMWZ1}): %(see page 90, \cite{GMWZ1})
%states that
\begin{equation}\label{max}\begin{aligned} (\gamma+\rho^2)|u_H|^2_{L^2}+|u_{P_+}|^2_{L^2} &+\rho^2|u_{P_-}|^2_{L^2}+ |u_H(0)|^2+|u_{P_+}(0)|^2+\rho^2|u_{P_-}(0)|^2\\&\lesssim
|\wprod{SF_{P_+},u_{P_+}}|+|\wprod{SF_{P_-},u_{P_-}}|+|\wprod{SF_{H},u_{H}}|
%\\&\lesssim
%\wprod{|F_{P_+}|,|u_{P_+}|}+\rho^2\wprod{|F_{P_-}|,|u_{P_-}|}+\wprod{|F_{H}|,|u_{H}|}
\end{aligned}\end{equation}
where $S$ is the degenerate symmetrizer constructed in \cite{GMWZ1}
(see equations (8.2),(8.2), and (6.18)) as follows
\begin{equation}\label{Sform}S = \begin{pmatrix} S_P
&0\\0&S_H\end{pmatrix}\end{equation} and
\begin{equation}\label{Spform}S_P = \begin{pmatrix} K ~Id%\mbox{Id}
&0\\0&-\rho^2\end{pmatrix}\end{equation} for sufficiently large
constant $K$ (independent of small parameter $\rho$); here, $Id$ is
the identity matrix and the two subblocks in $S_P$ have the same
sizes as those of $P_\pm$ in \eqref{para-bl}, correspondingly. Here
and throughout the paper, by $f \lesssim g$, we mean $f \le C g$,
for some positive constant $C$ independent of $\rho$.

There are two possibly subtle points in quoting \eqref{max} that we
would like to %make,
point out, namely, (i) the estimate \eqref{max} was proved in
Section 8, \cite{GMWZ1}, under the assumption (H4), but not under
the relaxed Hypothesis (H4'), and (ii) the estimate was obtained
only for the Lax shock case. However, in the first matter, the
variable multiplicity assumption is only involved in the hyperbolic
part (the $H$ block in \eqref{1steqsZ}) and the parabolic blocks
$P_\pm$ remain the same. Thus, the degenerate Kreiss-type
symmetrizers techniques (only involved in the parabolic blocks)
introduced in
\cite{GMWZ1} %are still applied here.
can still be applicable here. For the hyperbolic part, we now use
the recent construction of Kreiss-type symmetrizers in \cite{GMWZ6}
that applies to the relaxed Hypothesis (H4'), thus yielding the
$L^2$ estimate for this block. In dealing with the second matter, we
recall that %one of crucial steps in the analysis of \cite{GMWZ1} for
a crucial step in the analysis of \cite{GMWZ1} for the Lax shock
case was to proving the %correct
``right'' degeneracy of the boundary operator or Lemma 7.1 in
\cite{GMWZ1}, connecting with the Evans stability condition (D). We
then observe that with slight modification of the proof, the lemma
remains unchanged for the under/over--compressive shock
case, yielding the same result. % for these %new
%cases.
For sake of completeness, we shall recall the proof of Lemma 7.1,
\cite{GMWZ1}, with a straightforward extension to other cases
%rather than
than the Lax case in Appendix A.

In other words, with our above observations, %it is safe to
we may use the estimate \eqref{max} as stated under our current
assumptions in
treating all three types of shocks. %under the consideration.
In addition, thanks to that fact that the symmetrizer $S$ is
degenerate with order $\rho^2$ in the block $P_{-}$ (see
\eqref{Sform},\eqref{Spform} above), we can further estimate
\eqref{max} as
\begin{equation}\label{max-est}\begin{aligned} (\gamma+\rho^2)|u_H|^2_{L^2}+|u_{P_+}|^2_{L^2} &+\rho^2|u_{P_-}|^2_{L^2}+ |u_H(0)|^2+|u_{P_+}(0)|^2+\rho^2|u_{P_-}(0)|^2
\\&\lesssim
\wprod{|F_{P_+}|,|u_{P_+}|}+\rho^2\wprod{|F_{P_-}|,|u_{P_-}|}+\wprod{|F_{H}|,|u_{H}|}
.\end{aligned}\end{equation}

We note that in %the
a final step %there
in \cite[equation (8.11)]{GMWZ1}, the standard Young's inequality
%has been
was used to absorb all terms of $(u_H,u_P)$ into the left-hand side,
leaving the $L^2$ norm of $F$ alone in the right hand side. For our
purpose, we shall keep it as stated in \eqref{max-est}. %Here, by $f
%\lesssim g$, we mean $f \le C g$, for some $C$ independent of small
%parameter $\rho$.

We remark also that as shown in \cite{GMWZ1}, all of coordinate
transformation matrices are uniformly bounded. Thus a bound on $Z =
(u_H,u_P)^T$ would yield a corresponding bound on the solution $U$.

\subsection{$L^1\to L^p$ estimates}\label{sec-L1Lp} We establish the $L^1\to L^p$ resolvent bounds for solutions
of eigenvalue equations $(L_\txi-\lambda)U = f$ in the low frequency
regime; specifically, we are interested in regime of parameters
restricting to the surface
\begin{equation}\label{gamma}\Gamma^{\tilde \xi}:=\{\lambda~:~\R\lambda =
-\theta_1(|\tilde \xi|^2 + |\I\lambda|^2)\},\end{equation} for
$\theta_1>0$ and $|(\tilde \xi,\lambda)|$ sufficiently small. The
curve $\Gamma^\txi$ was introduced in \cite[equation (4.26)]{Z2}.
Introducing $\Gamma^\txi$ is in fact regarded as a key to the
analysis of long-time stability in multidimensions. The main point
here is that even though $\lambda$ enters into the stable complex
half-plane ($\{\R \lambda <0\}$), $\Gamma^\txi$ remains outside of
the essential spectrum of limiting linearized operators
$L_{\txi,\pm}$; see \cite[Lemma 2.21]{Z3}.

In addition, in a related matter, we would like to recall that the
Kreiss' symmetrizers constructed by O. Gu\`es, G. M\'etivier, M.
Williams, and K. Zumbrun can be attained in a full neighborhood of
basepoint $(\underline \xi,\underline \lambda)$ even for
$\Re\underline \lambda=0$ (see, e.g., Theorem 3.7, \cite{GMWZ6}).
Thus, the $L^2$ estimate \eqref{max-est} is in fact still valid in
any region of $$\gamma \ge - \theta (|\tau|^2 + |\txi|^2)$$ for
$\theta$ sufficiently small. In particular, we shall use
\eqref{max-est} for $\lambda$ restricted on the curve $\Gamma^\txi$.

We obtain the following:

\begin{proposition}[Low-frequency bounds]\label{prop-resLF} Under the hypotheses of Theorem \ref{theo-nonlin}, for $\lambda \in \Gamma^{\tilde \xi}$ and $\rho :=|(\tilde
\xi,\lambda)|$, $\theta_1$ sufficiently small, there holds the
resolvent bound \begin{equation}\label{res-bound} |(L_{\tilde
\xi}-\lambda)^{-1}\partial_{x_1}^\beta f|_{L^p(x_1)} \le
C\rho^{-3/2+(1-\alpha)\beta}| f|_{L^1(x_1)},\end{equation} for all
$2\le p\le \infty$, $\beta =0,1$, and $\alpha$ defined as in
\eqref{alpha}.
\end{proposition}

\begin{proof} %We first note that the earlier analysis of obtaining the $L^2$ stability estimate
%under the standard assumption $\R\lambda\ge0$ is still valid under a
%more general requirement that $\lambda \in \Gamma^{\tilde \xi}$ for
%sufficiently small $\theta_1>0$.

Changing variables as above subsection and taking the inner product
of each equation in \eqref{1steqsZ} against $u_H$ and $u_{P_\pm}$,
respectively, and integrating the results over $[0,x_1]$, for
$x_1>0$, we obtain
\begin{equation}\begin{aligned} \frac 12 |u_H(x_1)|^2 &= \frac 12|u_H(0)|^2 +
\R\int_0^{x_1}(H(\lambda,\txi)u_H\cdot u_H  + F_H\cdot u_H)dz,
\\\frac 12 |u_{P_\pm}(x_1)|^2 &= \frac 12|u_{P_\pm}(0)|^2 + \R\int_0^{x_1}({P_\pm}(\lambda,\txi)u_{P_\pm}\cdot u_{P_\pm}
 + F_{P_\pm}\cdot u_{P_\pm})dz.\end{aligned}\end{equation}

This together with use of Young's inequality into the last terms
involved in $F$ and the facts that $|H|\le C\rho$ and $|P_\pm|\le C$
yields
\begin{equation}\label{key-est}\begin{aligned} |u_H|_{L^\infty(x_1)}^2 &\lesssim |u_H(0)|^2 + \rho|u_H|_{L^2}^2 + |F_H|^2_{L^1},
\\|u_{P_\pm}|_{L^\infty(x_1)}^2 &\lesssim |u_{P_\pm}(0)|^2 + |u_{P_\pm}|_{L^2}^2+|F_{P_\pm}|^2_{L^1}.\end{aligned}\end{equation}

We are now in position of applying the $L^2$ stability estimate
\eqref{max-est}. In \eqref{key-est}, multiplying both sides of
equations of $u_H$ by $\rho$, of $u_{P_+}$ by $1$, and of $u_{P_-}$
by $\rho^2$, adding up results, and applying \eqref{max-est}, we
obtain
\begin{equation}\label{key-est1}\begin{aligned}\rho^2(|u_H|^2_{L^2}&+|u_P|^2_{L^2}) + \rho|u_H|_{L^\infty}^2+|u_{P_+}|_{L^\infty}^2
+\rho^2|u_{P_-}|_{L^\infty}^2
\\&\lesssim
\wprod{|F_{P_+}|,|u_{P_+}|}+\rho^2\wprod{|F_{P_-}|,|u_{P_-}|}+\wprod{|F_{H}|,|u_{H}|}
+|F_H|_{L^1}^2 + |F_P|_{L^1}^2,\end{aligned}\end{equation}(noting
that $\rho$ is assumed to be small; in particular, $\rho \le 1$.)

Applying again the standard Young's inequality:
$$\begin{aligned}
\langle|F_{H}|&,|u_{H}|\rangle+\wprod{|F_{P_+}|,|u_{P_+}|}+\rho^2\wprod{|F_{P_-}|,|u_{P_-}|}\\&\lesssim
\epsilon\Big[\rho|u_H|_{L^\infty}^2+|u_{P_+}|_{L^\infty}^2+\rho^2|u_{P_-}|_{L^\infty}^2\Big]
+C_\epsilon\Big[\rho^{-1}|F_H|_{L^1}^2+|F_{P_+}|_{L^1}^2+\rho^2|F_{P_-}|_{L^1}^2\Big]
\end{aligned}$$
with $\epsilon>0$ being sufficiently small, from \eqref{key-est1},
we easily arrive at
\begin{equation}\label{key-est2}\begin{aligned}\rho^2(|u_H|^2_{L^2}&+|u_P|^2_{L^2}) + \rho|u_H|_{L^\infty}^2+|u_{P_+}|_{L^\infty}^2
+\rho^2|u_{P_-}|_{L^\infty}^2
\\&\lesssim
\rho^{-1}|F_H|_{L^1}^2+|F_{P_+}|_{L^1}^2+\rho^2|F_{P_-}|_{L^1}^2+|F_H|_{L^1}^2
+ |F_P|_{L^1}^2.\end{aligned}\end{equation}

Therefore in term of $Z = (u_H,u_P)^t$, simplifying the above yields
\begin{equation}\label{estZ}\begin{aligned}\rho^2|Z|_{L^2(x_1)}^2+ \rho^2|Z|_{L^\infty(x_1)}^2\le
C\rho^{-1}|F|_{L^1}^2\end{aligned}\end{equation}Now from the change
of variables $Z = V^{-1}\Phi^{-1}\tilde W$, we have the same
estimates for $\tilde W$ and thus $U$, because all coordinate
transformation matrices are uniformly bounded. Hence, we also obtain
bounds \eqref{estZ} for $U$ or by the interpolation inequality:
\begin{equation}\label{res-boundU} |U|_{L^p(x_1)} \le
C\rho^{-3/2}| f|_{L^1(x_1)}.\end{equation} This thus proves the
proposition in the case of $\beta=0$.

%Following \cite{Z2,Z3}, define the curves
%$$(\txi,\lambda)(\rho,\hat\txi,\hat \tau):=(\rho \hat \txi,\rho i\hat \tau-
%\theta_1\rho^2),$$ where $\hat\txi \in \RR^{d-1}, \hat\tau \in \RR$
%and $(\hat\txi,\hat \tau)\in S^d$: $|\hat\txi|^2 +|\hat \tau|^2 =1$.
%As $(\rho,\hat\txi,\hat \tau)$ range in the compact set
%$[0,\delta]\times S^d$, $(\txi,\lambda)$ traces out the portion of
%the surface $\Gamma^\txi$ contained in the set $|\txi|^2 +
%|\lambda|^2 \le \delta$. Thus, estimate \eqref{res-boundU} yields

For $\beta =1$, we expect that $\partial_{x_1} f$ plays a role as
``$\rho f$'' forcing. Recall that the eigenvalue equations $(L_\txi
- \lambda)U =\partial_{x_1}f$ read
\begin{equation}\label{eg-eqs}\begin{aligned}\overbrace{(B^{11}U_{x_1})_{x_1}-(A^1U)_{x_1}}^{L_0U} - &i\sum_{j\not=1}A^j\xi_jU +
i\sum_{j\not=1}B^{j1}\xi_jU_{x_1}
\\&+i\sum_{k\not=1}(B^{1k}\xi_kU)_{x_1} -
\sum_{j,k\not=1}B^{jk}\xi_j\xi_k U - \lambda U =\partial_{x_1}f.
\end{aligned}\end{equation}

Now modifying the nice argument of Kreiss--Kreiss presented in
\cite{KK,GMWZ1}, we write $U = V + U_1$, where $V$ satisfies
\begin{equation}\label{auxeqs} (L_0-\lambda_0) V =
\partial_{x_1} f,\qquad x_1\in \RR,\end{equation}
for $\lambda_0 = \rho$. Noting that $A^1$ and $B^{11}$ depend on
$x_1$ only, we thus can apply here the one--dimensional Green kernel
bounds investigated by C. Mascia and K. Zumbrun as follows.

Let $G^0_{\lambda_0}$ be the Green kernel of $\lambda_0 - L_0$.
Observe that our assumptions as projected on one--dimensional
situations (i.e., $\tilde \xi =0$) are still the same as those in
\cite{Z3}. Thus, we apply Proposition 4.22 in \cite{Z3} for
\eqref{auxeqs}, noting that $\lambda_0=\rho$ is sufficiently small.
After a simplification, we simply obtain
\begin{equation} \label{ptbounds} |G^0_{\lambda_0}(x_1,y_1)|
\le C [\rho^{-1}e^{-\theta
|x_1|}e^{-\rho|y_1|}+e^{-\rho|x_1-y_1|}],\end{equation} and
\begin{equation} \label{ptbounds-der}
|\partial_{y_1}G^0_{\lambda_0}(x_1,y_1)| \le C [\rho^{-1}e^{-\theta
|x_1|}(\rho e^{-\theta \rho |y_1|}+\alpha
e^{-\theta|y_1|})+e^{-\rho|x_1-y_1|}(\rho +\alpha
e^{-\theta|y_1|})],\end{equation} where $\alpha$ is defined as in
\eqref{alpha}. We would like to remark here that %the integrated
%equations of the shock profiles are not well-posed in the
%undercompressive case, and thus
Lemma 5.23, \cite{Z3}, gives the estimate \eqref{ptbounds-der} with
$\alpha=0$ only for the Lax or overcompressive shocks. For the
undercompressive shocks, we must have the weaker bound by the term
$e^{-\theta|y_1|}$, that is, $\alpha=1$ (for further discussion,
see, e.g., equations below (5.106), \cite{Z3}).

Hence, using \eqref{ptbounds-der} and applying the standard
Hausdorff-Young's inequality, we obtain
\begin{equation}\label{Vbound} |V|_{L^p(x_1)}+ |V_{x_1}|_{L^p(x_1)}
\lesssim |f|_{L^1(x_1)} + \alpha \rho^{-1}|f|_{L^1(x_1)}\lesssim
\rho^{-\alpha}|f|_{L^1(x_1)} ,\end{equation} for all $1\le p\le
\infty$ and $\alpha=0$ or $1$ defined as in \eqref{alpha}.

Now from $U_1 = U-V$ and equations of $U$ and $V$, we observe that
$U_1$ satisfies
\begin{equation}\label{U1eqs}\begin{aligned}(L_\txi  - \lambda )U_1 = L(V,V_{x_1}),
\end{aligned}\end{equation} where $$ \begin{aligned}L(V,V_{x_1}) :&=
i\sum_{j\not=1}A^j\xi_jV - i\sum_{j\not=1}B^{j1}\xi_jV_{x_1}
-i\sum_{k\not=1}(B^{1k}\xi_kV)_{x_1} \\&\qquad \qquad+
\sum_{j,k\not=1}B^{jk}\xi_j\xi_k V + (\lambda-\lambda_0) V\\&=
\rho\cO(|V|+|V_{x_1}|).\end{aligned}$$

Therefore applying the result which we just proved for $\beta=0$ to
the equations \eqref{U1eqs}, we obtain
\begin{equation}\label{U1bound}\begin{aligned}|U_1|_{L^p(x_1)} &\le
C\rho^{-3/2}|L(V,V_{x_1})|_{L^1(x_1)}\le C\rho^{-3/2}\rho
\Big[|V|_{L^1}+|V_{x_1}|_{L^1}\Big] \\&\le
C\rho^{-3/2+(1-\alpha)}|f|_{L^1(x_1)}.\end{aligned}\end{equation}

Bounds \eqref{Vbound}, \eqref{U1bound} on $V$ and $U_1$ clearly give
our claimed bounds on $U$ by triangle inequality: $$|U|_{L^p} \le
|V|_{L^p}+|U_1|_{L^p}.$$

We thus obtain the proposition.
\end{proof}

\begin{remark}\label{rem-BS}\textup{Under the general structural assumptions,
our proof of the $L^1\to L^p$ bounds above depends only on the $L^2$
maximal estimate \eqref{max-est}. As the GMWZ theory covers to a
more general case than (H4$'$), namely, the (BS) condition
(Definition 4.9, \cite{GMWZ6}), our results thus apply to this case
as well without any additional work.}
\end{remark}

\begin{remark}\textup{In Appendices \ref{app-auxi} and \ref{indep-G}, we will
prove a slightly-weaker resolvent estimate like \eqref{res-bound} in
which the Green kernel bounds \eqref{ptbounds}, \eqref{ptbounds-der}
will not be used. Thus, our main results can in fact be derived
completely independent of the pointwise Green function estimates. }
\end{remark}

\subsection{Estimates on the solution operator}\label{sec-estS1} In this subsection, we complete the proof of Proposition \ref{prop-estS}. As mentioned earlier, it suffices to prove the bounds for $\cS_1(t)$, where the low frequency solution
operator $\cS_1(t)$ is defined as
\begin{equation}\label{cS1}\cS_1(t):=\frac{1}{(2\pi i)^d}\int_{|\tilde \xi|\le
r}\oint_{\Gamma^\txi\cap \{|\lambda|\le r\}} e^{\lambda t + i\tilde
\xi \cdot\tilde x}(L_{\txi} - \lambda)^{-1} d\lambda
d\txi.\end{equation}

\begin{proof}[Proof of bounds on $\cS_1(t)$] We first prove
\eqref{boundcS1} for $\beta=0$. Let $\hat u(x_1,\txi,\lambda)$
denote the solution of $(L_\txi-\lambda)\hat u = \hat f$, where
$\hat f(x_1,\txi)$ denotes Fourier transform of $f$, and
$$u(x,t):=\cS_1(t)f = \frac{1}{(2\pi i)^d}\int_{|\txi|\le r}\oint _{\Gamma^\txi\cap \{|\lambda|\le r\}}
e^{\lambda t+i\txi \cdot \tx}(L_\txi - \lambda)^{-1}\hat
f(x_1,\txi)d\lambda d\txi.$$

Using Parseval's identity, Fubini's theorem, the triangle
inequality, and Proposition \ref{prop-resLF}, we may estimate
$$\begin{aligned} |u|_{L^2(x_1,\tx)}^2(t) &=
\frac{1}{(2\pi)^{2d}}\int_{x_1} \int_{|\txi|\le
r}\Big|\oint_{\Gamma^\txi\cap \{|\lambda|\le r\}} e^{\lambda t}\hat
u(x_1,\txi,\lambda)d\lambda\Big|^2 d\txi dx_1
\\&\le
\frac{1}{(2\pi)^{2d}}\int_{\txi}\Big|\oint_{\Gamma^\txi\cap
\{|\lambda|\le r\}} e^{\R\lambda t}|\hat
u(x_1,\txi,\lambda)|_{L^2(x_1)}d\lambda\Big|^2 d\txi \\&\le
C|f|_{L^1(x)}^2\int_{\txi}\Big|\oint_{\Gamma^\txi\cap \{|\lambda|\le
r\}} e^{\R\lambda t}\rho^{-3/2}d\lambda\Big|^2 d\txi.
\end{aligned}$$

Specifically, parametrizing $\Gamma^\txi$ by $$\lambda(\txi,k) = ik
- \theta_1(k^2 + |\txi|^2), \quad k\in \RR$$ and noting that
$|d\lambda/dk|$ is bounded on $\Gamma^\txi \cap \{|\lambda|\le r\}$,
we estimate
$$\begin{aligned}
\int_{\txi}\Big|\oint_{\Gamma^\txi\cap \{|\lambda|\le r\}}
e^{\R\lambda t}\rho^{-3/2}d\lambda\Big|^2 d\txi &\le C
\int_{\txi}\Big|\int_\RR e^{-\theta_1(k^2+|\txi|^2)
t}\rho^{-3/2}dk\Big|^2 d\txi\\&\le
C\int_{\txi}e^{-2\theta_1|\txi|^2t}|\txi|^{-1-2\epsilon}\Big|\int_\RR
e^{-\theta_1k^2t}|k|^{\epsilon-1}dk\Big|^2 d\txi
\\&\le
Ct^{-(d-2)/2},
\end{aligned}$$ noting that $\int_{\RR^{d-1}} e^{-\theta |x|^2} |x|^{-\alpha} dx$ is finite, provided $\alpha < d-1$.

Similarly, parametrizing $\Gamma^\txi$ as above, we estimate
$$\begin{aligned} |u|_{L^\infty_{\tx, x_1}}(t)
&\le\frac{1}{(2\pi)^{d}} \int_{\txi}\oint_{\Gamma^\txi\cap
\{|\lambda|\le r\}} e^{\R\lambda t}|\hat
u(x_1,\txi,\lambda)|_{L^\infty(x_1)}d\lambda d\txi \\&\le
C|f|_{L^1(x)}\int_{\txi}\oint_{\Gamma^\txi\cap \{|\lambda|\le r\}}
e^{\R\lambda t}\rho^{-3/2}d\lambda d\txi\\&\le
\int_{\txi}e^{-\theta_1|\txi|^2 t}|\txi|^{-1/2-\epsilon}\int_\RR
e^{-\theta_1k^2t}|k|^{\epsilon-1}dkd\txi\\&\le
Ct^{-\frac{d-1}{2}+\frac 14}.
\end{aligned}$$
The $x_1-$derivative bounds follow similarly by using the version of
the $L^1\to L^p$ estimates for $\beta_1=1$, noting that in the
undercompressive case, both $\beta_1 = 0$ and $\beta_1=1$ have the
same bounds. The $\tx-$derivative bounds are straightforward by the
fact that $\widehat{\partial_{\tx}^{\tilde \beta} f} =
(i\txi)^{\tilde \beta} \hat f$.\end{proof}

\subsection{Proof of linearized stability} Applying estimates \eqref{boundcS1} and \eqref{boundcS2} on low- and high-frequency operators $\cS_1(t)$
and $\cS_2(t)$ obtained in Proposition \ref{prop-estS}, we obtain
\begin{equation} \begin{aligned} |U(t)|_{L^2} &\le
|\cS_1(t)U_0|_{L^2} + |\cS_2(t)U_0|_{L^2}\\&\le
C(1+t)^{-\frac{d-2}{4}}|U_0|_{L^1} + Ce^{-\eta t}|U_0|_{L^2}\\&\le
C(1+t)^{-\frac{d-2}{4}}|U_0|_{L^1\cap L^2}
\end{aligned}\end{equation}
and (together with the Sobolev embedding: $|f|_{L^\infty(\RR^d)}\le
C |f|_{H^s(\RR^d)}$ for $s>d/2$; see, for example, \cite[Lemma
1.4]{Z4})
\begin{equation}
\begin{aligned} |U(t)|_{L^\infty} &\le |\cS_1(t)U_0|_{L^\infty} +
|\cS_2(t)U_0|_{L^\infty}\\&\le C(1+t)^{-\frac{d-1}{2}+\frac
14}|U_0|_{L^1} + C|\cS_2(t)U_0|_{H^{[(d-1)/2]+2}}\\&\le
C(1+t)^{-\frac{d-1}{2}+\frac 14}|U_0|_{L^1} + Ce^{-\eta
t}|U_0|_{H^{[(d-1)/2]+2}}\\&\le C(1+t)^{-\frac{d-1}{2}+\frac
14}|U_0|_{L^1\cap H^{[(d-1)/2]+2}}.
\end{aligned}\end{equation}
These prove the bounds as stated in the theorem for $p=2$ and
$p=\infty$. For $2<p<\infty$, we use the interpolation inequality
between $L^2$ and $L^\infty$.
%\end{proof}

%%%%%%%%%%%%%%%%%%%%%%%%%%%%%%%%%%%%%%%%%%%%%%%%%%%%%%%%%%%%%%%%%%%%%%
\section{Nonlinear stability}\label{sec-stab}
Defining the perturbation variable $U:= \tilde U - \bU$, we obtain
the nonlinear perturbation equations
\begin{equation}\label{per-eqs} U_t - LU = \sum_j
Q^j(U,U_x)_{x_j},\end{equation} where
\begin{equation}\label{newqbounds}
\begin{aligned}
Q^j(U,U_x)&=\cO(|U||U_x|+|U|^2)\\
Q^j(U,U_x)_{x_j}&= \cO(|U||U_{x}|+|U||U_{xx}|+|U_x|^2)
\end{aligned}
\end{equation}
so long as $|U|$ remains bounded.

%\subsection{Lax or overcompressive case.} We give a proof of the
%main theorem for these cases.
\begin{proof}[Proof of Theorem \ref{theo-nonlin}] We prove the
theorem for the Lax or overcompressive case. The undercompressive
case follows very similarly. Define
\begin{equation}\label{zeta}
\begin{aligned}\zeta(t):=\sup_{0\le s\le t}
&\Big(|U(s)|_{L^2}(1+s)^{\frac{d-2}4}+|U(s)|_{L^\infty}(1+s)^{\frac
{d-1}2-\frac 14}\Big).\end{aligned}
\end{equation}

 We shall prove here that for all $t\ge
0$ for which a solution exists with $\zeta(t)$ uniformly bounded by
some fixed, sufficiently small constant, there holds
\begin{equation}\label{zeta-est}
\zeta(t) \le C(|U_0|_{L^1\cap H^s}+\zeta(t)^2) .\end{equation}

This bound together with continuity of $\zeta(t)$ implies that
\begin{equation}\label{zeta-est1} \zeta(t) \le 2C|U_0|_{L^1\cap H^s}\end{equation}
for $t\ge0$, provided that $|U_0|_{L^1\cap H^s}< 1/4C^2$. This
would complete the proof of the bounds as claimed in the theorem,
and thus give the main theorem.

By standard short-time theory/local well-posedness in $H^s$, and the
standard principle of continuation, there exists a solution $U\in
H^s$ on the open time-interval for which $|U|_{H^s}$ remains
bounded, and on this interval $\zeta(t)$ is well-defined and
continuous. Now, let $[0,T)$ be the maximal interval on which
$|U|_{H^s}$ remains strictly bounded by some fixed, sufficiently
small constant $\delta>0$. By an auxiliary energy estimate in
\cite[Proposition 5.9]{Z3} and the Sobolev embeding inequality
$|U|_{W^{2,\infty}}\le C|U|_{H^s}$ (again, see for example,
\cite[Lemma 1.4]{Z4}), we have
\begin{equation}\label{Hs}\begin{aligned}|U(t)|_{H^s}^2 &\le Ce^{-\theta t}|U_0|_{H^s}^2
+ C \int_0^t
e^{-\theta(t-\tau)}|U(\tau)|_{L^2}^2d\tau\\&\le
C(|U_0|_{H^s}^2+\zeta(t)^2)(1+t)^{-(d-2)/2}.
\end{aligned}\end{equation}
and so the solution continues so long as $\zeta$ remains small, with
bound \eqref{zeta-est1}, yielding existence and the claimed bounds.

Thus, it remains to prove the claim \eqref{zeta-est}. By Duhamel
formula
\begin{equation}\label{Duhamel}
\begin{aligned}
  U(x,t)=& \cS(t) U_0 + \int_0^t \cS(t-s)\sum_j
\partial_{x_j}Q^j(U,U_x)ds,
\end{aligned}
\end{equation} where $U(x,0) = U_0(x),$ we obtain
\begin{equation}\begin{aligned} |U(t)|_{L^2}\le& |\cS(t)U_0|_{L^2} +
\int_0^t|\cS_1(t-s)\partial_{x_j}Q^j(s)|_{L^2}ds+ \int_0^t
|\cS_2(t-s)\partial_{x_j}Q^j(s)|_{L^2}ds
\end{aligned}\end{equation}
where $|\cS(t) U_0|_{L^2}\le C (1+t)^{-\frac{d-2}{4}}|U_0|_{L^1\cap
L^2}$ as in the proof of linearized stability,
$$\begin{aligned}
\int_0^t|\cS_1(t-s)\partial_{x_j}Q^j(s)|_{L^2}ds &\le C\int_0^t (1+t-s)^{-\frac{d-2}{4}-\frac12}|Q^j(s)|_{L^1}ds
\\&\le C\int_0^t (1+t-s)^{-\frac{d-2}{4}-\frac12}|U|_{H^1}^2
ds\\&\le C(|U_0|_{H^s}^2+\zeta(t)^2)\int_0^t
(1+t-s)^{-\frac{d-2}{4}-\frac12}(1+s)^{-\frac{d-2}{2}}
\\&\le
C(1+t)^{-\frac{d-2}{4}}(|U_0|_{H^s}^2+\zeta(t)^2),
\end{aligned}$$
and
$$\begin{aligned} \int_0^t
|\cS_2(t-s)\partial_{x_j}Q^j(s)|_{L^2}ds&\le \int_0^t
e^{-\theta(t-s)}|\partial_{x_j}Q^j(s)|_{L^2}ds
\\&\le C\int_0^t
e^{-\theta(t-s)}|U|_{H^s}^2ds
\\&\le C(|U_0|_{H^s}^2+\zeta(t)^2) \int_0^t
e^{-\theta(t-s)}(1+s)^{-\frac{d-2}{2}}ds\\&\le
C(1+t)^{-\frac{d-2}{2}}(|U_0|_{H^s}^2+\zeta(t)^2).
\end{aligned}$$
Thus, dividing by $(1+t)^{-\frac{d-2}{4}}$, we obtain\begin{equation}\label{L2est}
|U(t)|_{L^{2}}(1+t)^{\frac{d-2}{4}} \le C(|U_0|_{L^1\cap
H^s}+\zeta(t)^2).\end{equation}

Similarly, we estimate the $L^\infty$ norm of $U$. By Duhamel's
formula \eqref{Duhamel}, we obtain
\begin{equation}\begin{aligned} |U(t)|_{L^\infty}\le& |\cS(t)U_0|_{L^\infty} +
\int_0^t|\cS_1(t-s)\partial_{x_j}Q^j(s)|_{L^\infty}ds\\&
+ \int_0^t |\cS_2(t-s)\partial_{x_j}Q^j(s)|_{L^\infty}ds
\end{aligned}\end{equation}where $|\cS(t) U_0|_{L^\infty}\le C (1+t)^{-\frac{d-1}{2}+\frac
14}|U_0|_{L^1\cap H^{[(d-1)/2]+2}}$,
$$\begin{aligned}\int_0^t|\cS_1(t-s)&\partial_{x_j}Q^j(s)|_{L^\infty}ds
\\&\le C\int_0^t (1+t-s)^{-\frac{d-1}{2}+\frac
14-\frac12}|Q^j(s)|_{L^1}ds
\\&\le C\int_0^t (1+t-s)^{-\frac{d-1}{2}+\frac
14-\frac12}|U|_{H^1}^2
\\&\le C(|U_0|_{H^s}^2+\zeta(t)^2)\int_0^t
(1+t-s)^{-\frac{d-1}{2}+\frac 14-\frac 12}(1+s)^{-\frac{d-2}{2}}
\\&\le
C(1+t)^{-\frac{d-1}{2}+\frac 14}(|U_0|_{H^s}^2+\zeta(t)^2)
\end{aligned}$$
and (by the Moser inequality; see, for example, inequality (1.22),
\cite{Z4}),
$$\begin{aligned} \int_0^t
|\cS_2(t-s)&\partial_{x_j}Q^j(s)|_{L^\infty}ds \\&\le\int_0^t
|\cS_2(t-s)\partial_{x_j}Q^j(s)|_{H^{[(d-1)/2]+2}}ds\\&\le \int_0^t
e^{-\theta(t-s)}|\partial_{x}Q^j(s)|_{H^{[(d-1)/2]+2}}ds
\\&\le C\int_0^t
e^{-\theta(t-s)}|U|_{L^\infty}|U|_{H^{[(d-1)/2]+4}}ds
\\&\le C(|U_0|_{H^s}^2+\zeta(t)^2) \int_0^t
e^{-\theta(t-s)}(1+s)^{-\frac{d-1}{2}+\frac
14}(1+s)^{-\frac{d-2}{4}}ds
\\&\le C(1+t)^{-\frac{d-1}{2}+\frac
14}(|U_0|_{H^s}^2+\zeta(t)^2).
\end{aligned}$$

Therefore we have obtained\begin{equation}
|U(t)|_{L^{\infty}}(1+t)^{\frac{d-1}{2}-\frac 14} \le
C(|U_0|_{L^1\cap H^s}+\zeta(t)^2)\end{equation} and thus completed
the proof of claim \eqref{zeta-est}, and the theorem.
\end{proof}

%%%%%%%%%%%%%%%%%%%%%%%%%%%%%%%%%%%%%%%%%%%%%%%%%%%%%%%%%%%%%%%%%%%%%%%%%%%%%%%%%%%%%%%%%%%%%%%%%%%%%%%%%%

\section{Two--dimensional case or cases with (H5)}\label{sec-H5}
In this section, we give %a %rather
%straightforward
a proof of Theorem \ref{theo-stabH5}. Again, notice that the only
assumption we make here that differs from those in \cite{Z3} is the
relaxed Hypothesis (H4'), treating the additional case of totally
nonglancing characteristic
roots, which is only involved in low--frequency estimates. That is to say, we %may
only need to establish the $L^1\to L^p$ bounds in low-frequency
regimes for this new case. We give the proof of these bounds by
modifying the proof in \cite[Section 12]{GMWZ1} and thus will not
cite the estimate \eqref{max-est} in this section; in fact, the
proof is completely independent of previous sections. In addition,
our proof is somewhat more direct and simpler than those in
\cite[Section 12]{GMWZ1} by not bypassing to the dual problem.

\begin{proposition}[Low-frequency bounds; \cite{Z3}, Corollary 5.11]\label{prop-resLFH5} Under the hypotheses of Theorem \ref{theo-stabH5},
for $\lambda \in \Gamma^{\tilde \xi}$ (see \eqref{gamma}) and $\rho
:=|(\tilde \xi,\lambda)|$, $\theta_1$ sufficiently small, there
holds the resolvent bound \begin{equation}\label{res-boundH5}
|(L_{\tilde \xi}-\lambda)^{-1}\partial_{x_1}^\beta f|_{L^p(x_1)} \le
C\gamma_2\rho^{\beta-1}| f|_{L^1(x_1)},\end{equation} for all $2\le
p\le \infty$, $\beta =0,1$, and $\gamma_2$ is the diagonalization
error (see \cite{Z3}, (5.40)) defined as
\begin{equation}\label{gamma2} \gamma_2 := 1+ \sum_{j,\pm}\Big[\rho^{-1}|\I\lambda
 - \eta_j^\pm(\txi)|+\rho\Big]^{1/s_j-1},\end{equation} with
 $\eta_j^\pm,s_j$ as in (H5).
\end{proposition}

%\begin{remark} The proposition is the
%
%\end{remark}

We again perform the standard procedure (see Section
\ref{GMWZresults}) of writing the linearized equations in form of
the first order eigenvalue equations \eqref{1steqsZ}:
\begin{equation}\label{resol-eqs}\partial_{x_1}\begin{pmatrix}U_H\\U_{P}\end{pmatrix}
= \begin{pmatrix}H&0\\0&{P}\end{pmatrix}
\begin{pmatrix}U_H\\U_{P}\end{pmatrix} +
\begin{pmatrix}F_H\\F_{P}\end{pmatrix}, \quad \bar \Gamma U =
0.\end{equation}

Locally, in a neighborhood of a base point
$X_0:=(\underline{\zeta},0)$ with $\zeta = (\tau,\gamma,\txi)$ and
$\lambda = \gamma + i\tau$, we further use the Assumption (H4') to
write $H$ in block--diagonal structure (see \cite[Proposition
6.1]{GMWZ1}) with appearance of a new mode, totally nonglancing, and
decompose the resolvent solution $U$ into
\begin{equation}U = U_{P} + U_{H_e} + U_{H_{h}} + U_{H_{g}}
+ U_{H_{t}},\end{equation}corresponding to parabolic, elliptic,
hyperbolic, glancing, or totally nonglancing blocks. We further
write $$U_{i} = U_{i+} + U_{i-}$$ for $i = P,H_e,H_h,H_g,H_t$, where
$U_{i\pm}$ are defined as the projections of $U_i$ onto the growing
(resp. decaying) eigenspaces of $\cG_+$ in \eqref{eq-bl} with
respect to the corresponding blocks.

These first four blocks have been treated in \cite[Corollary
5.11]{Z3} or \cite[Corollary 12.2]{GMWZ1} for which the totally
nonglancing modes are absent. For sake of completeness, we treat
these modes again here in a slightly different analysis, modifying
those of \cite[Section 12]{GMWZ1}. In fact, since each mode
interacts with the other via the Evans condition (D) or, more
precisely, the boundary estimate \eqref{bd-est}, we cannot obtain
\eqref{res-boundH5} for each mode separately.

We shall use the following simple lemma.

\begin{lemma}\label{lem-bounds} Let $U$ be a solution of $\partial_z U = QU + F$ with $U(+\infty)=0$. Assume that there is a positive
[resp., negative] symmetric matrix $S$ such that
\begin{equation}\label{Sym-ineq}\Re SQ := \frac12 (SQ + Q^* S^*)\ge \theta
Id\end{equation} for some $\theta>0$, and $S\ge Id$ [resp., $-S\ge
Id$]. Then there holds
\begin{equation}\label{est-posneg}\begin{aligned} |U|^2_{L^\infty}+
\theta|U|_{L^2}^2 &\lesssim
|F|_{L^1}^2\\
\mbox{[resp.,   } |U|^2_{L^\infty} + \theta|U|_{L^2}^2&\lesssim
|U(0)|^2 + |F|_{L^1}^2\mbox{ ]}.\end{aligned}\end{equation}
\end{lemma}
\begin{proof} Taking the real part of the inner product of the
equation of $U$ against $SU$ and integrating the result over
$[x_1,\infty]$ for the first case [resp., $[0,x_1]$ for the second
case], we easily obtain the lemma.\end{proof}

%
%
%We prove the bounds for totally nonglancing modes $U_{H_t}$, where
%\begin{equation}\label{nonglancing-eq}\partial_{x_1}U_{H_t} = Q_t U_{H_t} + F_{H_t}.\end{equation}

\begin{proof}[Proof of Proposition \ref{prop-resLFH5}] As
in \cite[Section 12.2]{GMWZ1}, the first step is to put blocks into
a diagonal form; indeed, parabolic blocks are already diagonal as in
\eqref{para-bl}; hyperbolic blocks are $1\times 1$ blocks with real
part vanishing at the base point $X_0$, but with real part $>0$
(resp. $<0$) when $\rho>0$ in polar coordinates (thus, vanishing at
order $\rho^2$ in original coordinates); elliptic blocks are those
$Q_k$ with $\R Q_k$ positive or negative definite at the base point
(thus, vanishing at order $\rho$); and finally glancing blocks are
of size larger than $1\times 1$ whose components are purely
imaginary at the base point. We recall the following lemma in
\cite{GMWZ1}, diagonalizing these glancing blocks.

%$$u'_{H_g}:=T^{-1}_{H_{g}} u_{H_g},$$ where $u_{H_g}:=u_{H_{g+}} + u_{H_{g-}}$. Here $u_{H_{g\pm}}$
%are defined as the projections of $u_{H_g}$ onto the growing (resp.
%decaying) eigenspaces of $Q_k ({\hat \zeta}, \rho)$ in
%\eqref{gl-block1}. We recall the following whose proof can be found
%in \cite{Z2,Z3} or Lemma 12.1, \cite{GMWZ1}.
\begin{lemma}[Lemma 12.1, \cite{GMWZ1}]\label{lem-gmwz} Diagonalize the glancing blocks $Q_k$ by
the transformation $T_{H_g}$, where $T_{H_g}$ may be chosen so that
\begin{equation}\label{T-est} |T_{H_g}|\le C, \qquad |T^{-1}_{H_g}|\le C\gamma_2,\qquad |T^{-1}_{{H_g}|_{H_{g-}}}|\le
C\gamma_1\end{equation} where $\gamma_2$ is defined as in
\eqref{gamma2} and $\gamma_1$ is defined as
\begin{equation}\label{gamma1}\gamma_1:= \max _k\Big[\rho^{-1}|\I\lambda
 - \eta_k^\pm(\txi)|+\rho\Big]^{(1-[(\nu_k+1)/2])/\nu_k},\end{equation} and
$T^{-1}_{{H_g}|_{H_{g-}}}$ denotes the restriction of $T^{-1}_{H_g}$
to subspace $H_{g-}$. %In particular, $\gamma_2 \gamma_1^{-2} \ge 1$.

In addition, after a further transformation if necessary,
\begin{equation} Q'_k:=T^{-1}_{H_g}Q_k T_{H_g} =
 \mbox{\textup{diag}}(\alpha_{k,1},\cdots, \alpha_{k,l},\alpha_{k,l+1},\cdots ,\alpha_{k,\nu_k})\end{equation} with
 \begin{equation}\label{rea-est}\begin{aligned}-&\gamma_1^{-2}\R ~\alpha_{k,j} \ge C\rho^2, \quad j = 1,...,l,\\
 &\gamma_1^{-2}\R ~\alpha_{k,j} \ge C\rho^2 , \quad j = l+1,...,\nu_k.\end{aligned}\end{equation}

\end{lemma}

\begin{remark} \textup{$\gamma_1,\gamma_2$ are identical to $\alpha,\beta$
in \cite{GMWZ1}, respectively, and \eqref{rea-est} was calculated in
\cite[equation (12.40)]{GMWZ1}.}
\end{remark}

We now can work in diagonalized coordinates: $$U':=T^{-1}_{H_g} U$$
where $T_{H_g}$ are obtained as in Lemma \ref{lem-gmwz} for glancing
blocks and identity matrices for the other blocks. In these
coordinates, since blocks are diagonal and growing/decaying
subspaces (at least for the first four modes) are separated, we
apply Lemma \ref{lem-bounds} for each block with $S = \pm Id$,
yielding
\begin{equation}\label{blocks-est}\begin{aligned} |U'_{i+}|^2_{L^\infty}+ \theta_i|U'_{i+}|_{L^2}^2
&\lesssim |F'_{i}|_{L^1}^2,
\\|U'_{i-}|^2_{L^\infty} + \theta_i|U'_{i-}|_{L^2}^2&\lesssim |U'_{i-}(0)|^2
+ |F'_{i-}|_{L^1}^2,\end{aligned}\end{equation} where $\theta_i =
1,\rho,\rho^2,\min_j|\R ~\alpha_{k,j}|$ for $i=P,H_e,H_h,H_g$.

For the totally nonglancing blocks $Q_t^k$, as constructed in
\cite{GMWZ6}, Lemma 5.3, there exist symmetrizers $S^k$ that are
definite positive [resp., negative] when the mode is totally
incoming [resp., outgoing]. Denote $U'_{H_{t+}} $ [resp.,
$U'_{H_{t-}}$] associated with totally incoming [resp. outgoing]
modes. Then by applying Lemma \ref{lem-bounds} with $\theta =
\rho^2$, we obtain
\begin{equation}\label{totalnongl-est}\begin{aligned} |U'_{H_{t+}}|^2_{L^\infty}+ \rho^2|U'_{H_{t+}}|_{L^2}^2 &\lesssim |F'_{H_{t+}}|_{L^1}^2,
\\|U'_{H_{t-}}|^2_{L^\infty} + \rho^2|U'_{H_{t-}}|_{L^2}^2&\lesssim |U'_{H_{t-}}(0)|^2
+ |F'_{H_{t-}}|_{L^1}^2.\end{aligned}\end{equation}

To finish the proof, we only need to deal with the boundary terms
$|U'_{i-}(0)|^2$ and $|U_{H_{t-}}(0)|^2$ in
\eqref{blocks-est},\eqref{totalnongl-est}. We could use a more
detailed version of the $L^2$ stability estimate \eqref{max-est},
corresponding to each diagonal blocks (see \cite{GMWZ6}), yielding
bounds on these boundary terms. However, let us now follow the
boundary treatment presented in \cite[Section 12.3]{GMWZ1} instead,
being rather independent of \eqref{max-est}.

The diagonalized boundary condition is $\Gamma':=\Gamma T_{H_g}$. By
computing, we observe that
$$ |\Gamma' U'_{H_{g-}}| =
|\Gamma U_{H_{g-}}| \ge C^{-1} |U_{H_{g-}}| \ge
\frac{C^{-1}|U'_{H_{g-}}|}{|T^{-1}_{{H_g}|_{H_{g-}}}|} \ge
C^{-1}\gamma_1^{-1}|U'_{H_{g-}}|.$$

Thus, together with \eqref{bd-est}
\begin{equation}\label{newbd-est}|\Gamma' U'_{-}| = |\Gamma U_-|\ge
C^{-1}\Big[|U'_{H_{e-}}|+|U'_{H_{h-}}|+|U'_{H_{t-}}|+\gamma_1^{-1}|U'_{H_{g-}}|+\rho|U'_{P_-}|\Big].\end{equation}

Meanwhile, we have at $x_1=0$
\begin{equation}\label{gmwztrick}|\Gamma' U'_{-}| \le  |\Gamma' U'| + |\Gamma' U'_{+}| \lesssim |U'_+(0)|\le |U'_+|_{L^\infty}.\end{equation}

Now, multiplying the first equations in \eqref{blocks-est},
\eqref{totalnongl-est} by a sufficiently large constant $k$ and the
second equations by $1,\gamma_1^{-2}$, or $\rho^2$, corresponding to
each block with its boundary degeneracy of order in
\eqref{newbd-est}, and adding up the results, we easily obtain
\begin{equation}\label{final-est}\begin{aligned}
&|U'_{P+}|^2_{L^\infty}+ |U'_{P+}|^2_{L^2}+|U'_{H+}|^2_{L^\infty}+
\rho^2|U'_{H+}|_{L^2}^2+\\&
\rho^2|U'_{P-}|_{L^\infty}^2+\rho^2|U'_{P-}|^2_{L^2}+\gamma_{1}^{-2}|U'_{H-}|_{L^\infty}^2+\rho^2|U'_{H-}|^2_{L^2}
\end{aligned}
\lesssim |F'|_{L^1}^2,
\end{equation}
(noting that $\rho\le 1$, $\rho\le\gamma_1^{-1}\le 1$, and
$\gamma_1^{-2}\min_j|\R ~\alpha_{k,j}| \gtrsim\rho^2$ by
\eqref{rea-est}). This yields
\begin{equation}\label{final-est1}\rho^2|U'|^2_{L^\infty}+ \rho^2|U'|_{L^2}^2 \lesssim
|F'|_{L^1}^2
\end{equation} or equivalently, \begin{equation}\label{final-est2}|U'|_{L^p}\lesssim
\rho^{-1}|F'|_{L^1}, \quad \forall ~p\ge 2.
\end{equation}
Thus, by recalling that $U= T_{H_g} U'$ and $F'=T^{-1}_{H_g} F$,
\eqref{final-est2} and \eqref{T-est} immediately yield the
proposition for $\beta=0$. For $\beta=1$, we can follow the
Kreiss--Kreiss trick as presented in the proof of Proposition
\ref{prop-resLF}, thus completing the proof of Proposition
\ref{prop-resLFH5}. \end{proof}

\begin{proof}[Proof of Theorem \ref{theo-stabH5}] Proposition \ref{prop-resLFH5} is the Corollary 5.1 in \cite{Z3}
with an extension to the totally nonglancing cases. Thus, we can now
follow word by word the proof in \cite{Z3}, yielding the theorem.
\end{proof}

\begin{remark} \textup{We have seen in the above argument that the existence
of positive/negative Kreiss' symmetrizers with an appropriate
constant $\theta$ (in Lemma \ref{lem-bounds}) is sufficient to
obtain the result. Though, proving the existence of such
symmetrizers is a highly-technical task in general for variable
multiplicity blocks. See \cite{GMWZ5,GMWZ6}. }
\end{remark}

\appendix

\section{Evans function for the doubled boundary
problem}\label{app-Evans}

For sake of completeness, we recall here the proof of Lemma 7.1,
\cite{GMWZ1} and its straightforward extension to %the over- and
the case of over- and under-compressive shocks.

Consider the $2N \times 2N$ doubled boundary problem
\eqref{1steqstW} (with $N :=n+r$)\begin{equation}\label{doubledeq}
\left\{\begin{array}{lcr} &U_x - G(x,\zeta) U = F,\\&\Gamma U= 0
\quad \mbox{on }x=0,\end{array}\right.\end{equation} where $U =
(U_+,U_-)$ and $\Gamma U = U_+ - U_-$, with $U_+ =
(U_1,\dots,U_{N}),U_- = (U_{N+1},\dots,U_{2N})$.

Let $\cE_-(\hat \zeta,\rho)$, for $\hat \gamma>0$ and $\rho>0$, be
the space of boundary values at $x = 0$ of decaying solutions to the
homogeneous problem
$$U_x -
G(x,\zeta) U = 0.$$

%As also mentioned in \cite{GMWZ1},
As shown, for example, in \cite[Theorem 3.7]{GMWZ6}, the space
$\cE_-(\hat \zeta,\rho)$ has a continuous extension to a small
neighborhood of $\hat \gamma = 0$, $\rho \ge 0$. Then the Evans
function for \eqref{doubledeq} is defined as the $2N\times 2N$
determinant:
\begin{equation}\label{Evans-doubled}\DD(\hat \zeta,\rho) = \det (\ker \Gamma,
\cE_-){|_{x=0}}.\end{equation}

Meanwhile, the Evans function $D_L$ for the problem \eqref{1steqsW}
is defined as
\begin{equation}\label{Evans1}D_L(\hat \zeta,\rho) = \det
(\CalU_1^R,\dots,\CalU_{k}^R,\CalU_{k+1}^L,\dots,\CalU_{N}^L){|_{x=0}}\end{equation}
which is analytic for $\R \lambda >0$ and can be continuously
extended to a small neighborhood of $\R\lambda =0$ (see, e.g., Lemma
5.24, \cite{Z3}). Now let $\phi_j$, $j=1,\dots,l$ be the derivative
of the profile $\bU^\delta$ with respect to $\delta_j$, where $l$ is
the dimension of the smooth manifold $\{\bU^\delta(\cdot)\}$ defined
as in (H3). Thanks to the Evans condition (D), without loss of
generality, we can assume that
\begin{equation} \label{normlization}\CalU_j^R (x,\hat \zeta,0)= \CalU^L_{N-j+1}(x,\hat \zeta,0) =
(\phi_j(x),0),\end{equation} for $j = 1,\dots,l$.

Let $e_j\in \CC^{N}$ be the unit vectors $$e_j =
\frac{(\phi_j(0),0)}{|\phi_j(0)|}, \quad j=1,\dots, l,$$ and extend
to an orthonormal basis $e_1,\dots,e_N$ of $\CC^N$. Then the Evans
function \eqref{Evans-doubled} for the doubled boundary value
problem can be explicitly defined as
\begin{equation}\label{Evans-doubled2}\DD(\hat \zeta,\rho) = \det \begin{pmatrix}
e_1 & \dots &e_N & \CalU_1^R &\dots &\CalU_k^R &0& \dots &0\\
e_1 & \dots &e_N & 0& \dots &0 &\CalU_{k+1}^L &\dots &\CalU_N^L
\end{pmatrix}{|_{x=0}}.\end{equation}

We also set \begin{equation}\label{Ephi} \cE_{-,\phi}(\hat
\zeta,\rho) = \mbox{span}\left\{
\begin{pmatrix} \CalU_{1}^R\\\CalU_{N}^L\end{pmatrix},\dots ,\begin{pmatrix}
\CalU_{l}^R\\\CalU_{N-l+1}^L\end{pmatrix} \right\} {|_{(0,\hat
\zeta,\rho)}}.
\end{equation}
For $\epsilon>0$ fixed, denote by $\cE^c_{-,\phi,\epsilon}(\hat
\zeta,\rho)$ any complementary subspace in $\cE_{-}(\hat
\zeta,\rho)$ varying continuously with $(\hat \zeta,\rho)$ such that
\begin{equation}\label{ap-unibound} \cE_{-}(\hat
\zeta,\rho) = \cE_{-,\phi}(\hat \zeta,\rho) \oplus
\cE^c_{-,\phi,\epsilon}(\hat \zeta,\rho)\end{equation} with
uniformly bounded projections for $0\le \rho\le\epsilon$.

Then, we recall the following proposition that was proved for the
Lax shock case in \cite{GMWZ1}, Proposition 7.1.

\begin{proposition}\label{GMWZprop}
(1) Let $D_L(\hat \zeta,\rho)$ and $\DD(\hat\zeta,\rho)$ be the
Evans functions defined as above. Then
\begin{equation} D_L(\hat \zeta,\rho) = (-1)^N \DD(\hat
\zeta,\rho).\end{equation}

(2) Under the Evans assumption (D), we have the following.

(a) For any choice of $0<\delta<R$ there is a constant $C_\delta,R$
such that when $\delta \le \rho\le R$, \begin{equation}|\Gamma u
|\ge C_{\delta,R}|u| \quad \mbox{ for  } u \in \cE_-(\hat
\zeta,\rho). \end{equation}

(b) There exist positive constants $C_1, C_2,\delta$ such that
\begin{equation}\label{Abound1} C_1\rho|u|\le |\Gamma u|\le C_2\rho |u|\qquad \mbox{ for  }
u \in \cE_{-,\phi}(\hat \zeta,\rho)\end{equation} for $0\le \rho\le
\delta$.

(c) There exists $C>0$ such that
\begin{equation}\label{ap-2c} |\Gamma u|\ge C|u|\qquad \mbox{ for  }
u \in \cE^c_{-,\phi,\epsilon}(\hat \zeta,\rho)\end{equation} for
$0\le \rho\le \epsilon$.

(d) For any choice of $R >0$ there is a constant $C_R$ such that for
$0\le \rho\le R$, \begin{equation}\label{ap-2d}|\Gamma u|\ge C_R\rho
|u|\qquad \mbox{ for  }u  \in \cE_{-}(\hat
\zeta,\rho).\end{equation}
\end{proposition}

\begin{proof} We follow word by word the proof for the Lax shock case in \cite{GMWZ1},
Proposition 7.1. First, by performing the row matrix operation, (1)
is clear. (2a) follows by continuity and compactness, and the fact
that $\Gamma u$ is nonvanishing for nonzero $u\in \cE_-(\hat
\zeta,\rho)$ when $\rho>0$ by Evans function assumption (D).

For the proof of (2b), let us denote the matrix in
\eqref{Evans-doubled2} by $\CalM$ and perform column operations to
replace the %$l^{th}$ last
last $l$ columns of $\CalM$ by $\begin{pmatrix}
\CalU_{j}^R\\\CalU_{N-j+1}^L\end{pmatrix}$, and call the resulting
matrix $\CalM_1$. Now thanks to the normalization
\eqref{normlization} and the fact that fast modes depend
analytically on $\rho$, we have for $j = 1,\dots,l$
\begin{equation}
\begin{pmatrix} \CalU_{j}^R\\\CalU_{N-j+1}^L\end{pmatrix}(0,\hat \zeta, \rho)
= \begin{pmatrix}(\phi_j(0),0)\\(\phi_j(0),0)\end{pmatrix} +
\begin{pmatrix}c_{1j}(\hat \zeta)\\c_{2j}(\hat \zeta)\end{pmatrix}
\rho + \cO(\rho^2).
\end{equation}

Thus, the definition of $e_j$, linearity of the determinant in the
last %$l^{th}$
$l$ columns, and the Evans condition (D) show that $c_{1j} -c_{2j}$
are nonzero for all $j$. This together with the definition of
$\Gamma U$
$$\Gamma \begin{pmatrix}U^R\\U^L\end{pmatrix} = U^R - U^L$$ %for any $\begin{pmatrix}U^R\\U^L\end{pmatrix}$
%being a linear combination of
%$\begin{pmatrix}U^R_{j}\\U^L_{N-j+1}\end{pmatrix}$,
yields
\eqref{Abound1} at once.

%The proof of (2c) and (2d) follows the same way as in \cite{GMWZ1}
%with the above minor change. Thus, we omit it here.

(2c) Let $v_1(\hat \zeta,\rho),\cdots,v_{N}(\hat \zeta,\rho)$ be the
last $2n$ columns of the matrix $\CalM_1$ defined above. These
vectors form a basis for $\cE_-(\hat \zeta,\rho)$. Take an arbitrary
vector $w\in \cE_{-,\phi,\epsilon}^c(\hat \zeta,\rho)$. Then
\begin{equation}\label{ap-w}w = \sum_{j=1}^{N}c_{j,\epsilon}(\hat \zeta,\rho)v_j(\hat
\zeta,\rho),\end{equation} where $c_{j,\epsilon}(\hat \zeta,\rho)$
depend continuously on $(\hat \zeta,\rho)$.

Set $c^{'}_\epsilon = (c_{1,\epsilon},\cdots,c_{N-l,\epsilon})$ and
$c^{''}_\epsilon = (c_{N-l+1,\epsilon},\cdots,c_{N,\epsilon})$. The
condition that the projections in \eqref{ap-unibound} are uniformly
bounded implies that there is an $\epsilon_0>0$ such that
\begin{equation}\label{ap-wb} |c{'}_\epsilon (\hat \zeta,\rho)|\ge \epsilon_0
|c^{''}_\epsilon(\hat \zeta,\rho)|,\end{equation} for $0\le \rho\le
\epsilon$.

In view of (D), we just need to show that $\Gamma w$ is nonvanishing
at $\rho=0$ for $w$ as in \eqref{ap-w} and \eqref{ap-wb} with
$|(c^{'}_\epsilon,c^{''}_\epsilon)|=1$, since \eqref{ap-2c} then
follows by continuity and compactness. Suppose $\Gamma w=0$ at
$(\hat \zeta,0)$ for some such $w$. Because of \eqref{ap-wb} some
$c_{j,\epsilon}$ with $j\le N-l$, say, $c_{1,\epsilon}$ satisfies
\begin{equation}|c_{1,\epsilon}|\ge c_0\end{equation}for $\rho $
near $0$, and for some $c_0>0$. Since $\Gamma w=0$ at $\rho =0$ and
$w(\hat \zeta,\rho)$ is continuous, we have
\begin{equation}\label{ap-w1} w(\hat \zeta,\rho) = \begin{pmatrix}a(\hat
\zeta) \\a(\hat \zeta)\end{pmatrix} + \cO(\rho).\end{equation} Write
$v_j = (v_{j+},v_{j-})$, use column operations to replace $v_1$ in
$\CalM_1$ by $w$, and call the resulting matrix $\CalM_2$. Then $
\CalM_2 =$
$$\begin{pmatrix}
e_1 & \dots &e_N & a(\hat \zeta)+\cO(\rho)&v_{2+} &\dots &v_{N-l,+} &(\phi_1(0),0)+\cO(\rho)& \dots &(\phi_l(0),0)+\cO(\rho)\\
e_1 & \dots &e_N & a(\hat \zeta)+\cO(\rho)& v_{2-} &\dots &v_{N-l,-}
&(\phi_1(0),0)+\cO(\rho)& \dots &(\phi_l(0),0)+\cO(\rho)
\end{pmatrix} .$$
\eqref{ap-w1} implies that $|\det \CalM_2(\hat \zeta,\rho)|\ge C
|\det \CalM_1(\hat \zeta,\rho)|$ for some $C>0$ uniformly near
$(\hat \zeta,0)$. But $$ \det \CalM_2(\hat \zeta,\rho) = \cO(\rho)^l
\cO(\rho) \qquad \mbox{as}\quad \rho \to 0.$$ This contradicts the
assumed vanishing of $\det \CalM = \det \CalM_1$ to exactly $l^{th}$
order at $\rho =0$.

(2d) For any fixed $(\hat \zeta,\rho)$, let $u^* =
\begin{pmatrix}u_+(\hat \zeta,\rho)\\u_-(\hat
\zeta,\rho)\end{pmatrix}  \in \cE_-(\hat \zeta,\rho)$ be an element
where the minimum $$\min_{|u|=1,u\in \cE_-(\hat \zeta,\rho)}|\Gamma
u|$$ is attained. Write $u^* =\sum_{j=1}^{N}c^*_{j,\epsilon}(\hat
\zeta,\rho)v_j(\hat \zeta,\rho)$ and define
$c^{*'}_{j,\epsilon},c^{*''}_{j,\epsilon}$ in the same way as above.
Then, again, the uniform boundedness of the projections in
\eqref{ap-unibound} implies that there is an $\epsilon_0>0$ such
that either
\begin{equation}\label{ap-ux1} |c{*'}_\epsilon (\hat \zeta,\rho)|\ge \epsilon_0
|c^{*''}_\epsilon(\hat \zeta,\rho)|\end{equation} or
\begin{equation}\label{ap-ux2} |c{*''}_\epsilon (\hat \zeta,\rho)|\ge
\epsilon_0 |c^{*'}_\epsilon(\hat \zeta,\rho)|\end{equation} for
$0\le \rho\le \epsilon$. Correspondingly, these imply that, without
loss of generality, there holds either
\begin{equation}\label{ap-ux3}|c^*_{1,\epsilon}|\ge
c_0\end{equation} or
\begin{equation}\label{ap-ux4}|c^*_{N,\epsilon}|\ge
c_0\end{equation} for $\rho $ near $0$, and for some $c_0>0$.

In the case that \eqref{ap-ux3} holds, as above, we perform column
operations to replace $v_1$ in $\CalM_1$ by $u^*$, and call the
result $\CalM_3$. Then $ \CalM_3 =$
$$\begin{pmatrix}
e_1 & \dots &e_N & u_+&v_{2+} &\dots &v_{N-l,+} &(\phi_1(0),0)+\cO(\rho)& \dots &(\phi_l(0),0)+\cO(\rho)\\
e_1 & \dots &e_N & u_-& v_{2-} &\dots &v_{N-l,-}
&(\phi_1(0),0)+\cO(\rho)& \dots &(\phi_l(0),0)+\cO(\rho)
\end{pmatrix} .$$
Next perform column operations to replace the column $u^* =
\begin{pmatrix}u_+\\u_-\end{pmatrix}$ by $$\begin{pmatrix} u_+ -
u_-\\0\end{pmatrix} = \begin{pmatrix} \Gamma u^*\\0\end{pmatrix}.$$
Thus, by direct calculations and \eqref{ap-ux3}, $$|\det \CalM_3 |=
|\Gamma u^*|\cO(\rho)^l \ge C|\det\CalM_1| = C|\det \CalM| \ge
C\rho^l,$$ which gives $|\Gamma u^*|\ge C$.

Similarly, in the case that \eqref{ap-ux3} holds, replacing $v_N$ in
$\CalM_1$ by $u^*$, denoting the resulting matrix by $\CalM_4$, and
performing column operations as above, we then obtain
$$|\det\CalM_4| = |u_+(\hat \zeta) - u_-(\hat \zeta,\rho)|
\cO(\rho)^{l-1} = |\Gamma u^*|\cO(\rho)^{l-1}.$$ This together with
$|\det\CalM_4| \ge C |\det\CalM_1| =\cO(\rho^l)$ by \eqref{ap-ux4}
yields $|\Gamma u^*|\ge C\rho $.

Thus, altogether we obtain $|\Gamma u|\ge C\rho|u|$, for $u\in
\cE_-(\hat \zeta,\rho)$ with $|u|=1$, uniformly in $\rho$ near $0$.
Together with (2a), this proves (2d).
\end{proof}

Now let $\cT$ be the MZ conjugation such that \eqref{doubledeq}
leads to the following constant--coefficient system
\begin{equation}\label{doubledeq-MZ}
\left\{\begin{array}{lcr} &U_x - G(\infty,\zeta) U = F,\\&\Gamma_1
U= 0 \quad \mbox{on }x=0,\end{array}\right.\end{equation} where
$\Gamma_1 = \Gamma \cT$ and $G$ has the block form as in
\eqref{eq-bl},\eqref{para-bl}:
\begin{equation} \label{G-bl}G(\infty,\zeta) =
\begin{pmatrix}P_+(\zeta)&0&0\\0&P_-(\zeta)&0\\0&0&H(\hat\zeta,\rho)\end{pmatrix}\end{equation}

Thus, we can decompose $U\in \CC^{2N}$ as follows
\begin{equation}\label{Udecompose}
U= U_{P_+} + U_{P_-} + U_{H_+} + U_{H_-},
\end{equation} and set $$U_- = U_{P_-} + U_{H_-} \in \cE_-(\hat
\zeta,\rho).$$

Define the $l$-dimensional subspace $E_{P_{1-}}$ of $E_{P-}$ by
$$\cE_{-,\phi} = \cT E_{P_{1-}},$$ where $\cE_{-,\phi}$ is defined
as in \eqref{Ephi}, and for $\epsilon>0$ fixed, chose a smoothly
varying complementary subspace $E_{P_{2-,\epsilon}}$ such that
\begin{equation}\begin{aligned} &E_{P_-} = E_{P_{1-}} (\hat
\zeta,\rho) \oplus E_{P_{2-,\epsilon}}(\hat \zeta,\rho),\\&U_{P_-} =
U_{P_{1-}} + U_{P_{2-,\epsilon}}\end{aligned}\end{equation} with
uniformly bounded projections for $0\le \rho\le\epsilon$. Take
\begin{equation} \cE^c_{-,\phi,\epsilon} = \cT (E_{P_{2-,\epsilon}} \oplus
E_{H_{-}}).\end{equation} $\cE^c_{-,\phi,\epsilon}$ is then a choice
that works in \eqref{ap-unibound}.

Then the following is an immediate consequence of Proposition
\ref{GMWZprop}.

\begin{corollary}\label{auxi-coro} There exist positive constants
$C_1,C_2$ and $\delta_0$ such that for $0\le \rho \le \delta_0$
\begin{equation}\label{auxiG1} \begin{aligned} (\mbox{a}) \qquad& C_1\rho
|U_{P_{1-}}|\le |\Gamma_1 U_{P_{1-}}| \le C_2 \rho |U_{P_{1-}}|,\\
(\mbox{b}) \qquad & |\Gamma_1( U_{H_{-}} + U_{P_{2-,\epsilon}})| \ge C_1 ( |U_{H_{-}} |+| U_{P_{2-,\epsilon}}|),\\
(\mbox{c}) \qquad & |\Gamma_1 U_{-}| \ge C_1 \rho
|U_{-}|,\end{aligned}\end{equation} where $\Gamma_1$ is defined as
in \eqref{doubledeq-MZ}. These estimates hold uniformly near the
basepoint $X_0=(\underline{\hat\zeta},0)$. \end{corollary}

Thus, we obtain the following lemma which is essential for the
construction of degenerate symmetrizers.

\begin{lemma}[Lemma 7.1, \cite{GMWZ1}] There exists a constant
$\delta>0$ such that for $\rho$ sufficiently small we have
\begin{equation}\label{bd-est}|\Gamma U_-| \ge \delta (|U_{H_-}| + \rho
|U_{P_-}|)\end{equation} uniformly in a neighborhood of the base
point $X_0=(\underline{\hat\zeta},0)$. %$\hat\zeta =
%(\hat\tau,\hat\gamma,\hat\txi)$
%here, recall that $U_{i\pm}$ are
%defined as the projections of $U_i$ onto the growing (resp.
%decaying) eigenspaces of $\cG_+$ in \eqref{eq-bl} with respect to
%the corresponding blocks $P$ or $H$.
\end{lemma}
\begin{proof} In view of \eqref{auxiG1} (a), (b), we have
$$\begin{aligned} |\Gamma_1U_-| &= |\Gamma_1 U_{H_-} + \Gamma _1 U_{P_{1-}}+ \Gamma _1
U_{P_{2-,\epsilon}}|\\&\ge C(|U_{H_-}| + |U_{P_{2-,\epsilon}}| ) -
C\rho |U_{P_{1-}}|.
\end{aligned}$$
Adding a sufficiently small multiple of this inequality to the
inequality \eqref{auxiG1} (c) $$|\Gamma_1 U_-| \ge C\rho |U_-| =
C\rho (|U_{H_-}| + |U_{P_{1-}}| +|U_{P_{2-,\epsilon}}| ),$$ we
obtain for $\rho$ small $$|\Gamma_1 U_-| \ge \delta(|U_-| +
\rho|U_{P_{1-}}| + |U_{P_{2-,\epsilon}}|),$$ which implies
\eqref{bd-est}.\end{proof}

\section{Auxiliary problem}\label{app-auxi} %: Symmetrizers-type technique}
In this section we consider the $n\times n$ system on the whole real
line $\RR$
\begin{equation}\label{auxi-eq} L_0V :=
(B^{11}V_{x})_{x}-(A^1V)_{x} = f_x\end{equation} where $A^1,B^{11}$
are same as in \eqref{eg-eqs}. Let us recall $B^{11} =
\begin{pmatrix}0&0\\b_1^{11}&b_2^{11}\end{pmatrix}$. We shall derive an estimate slightly similar to \eqref{Vbound} by
Kreiss-type symmetrizers techniques in the case of Lax and
overcompressive shocks. %Thus this would yield the results as claimed
%in Theorems \ref{theo-lin}-\ref{theo-stabH5} completely independent
%of pointwise Green function bounds.
This will be done by modifying
the proof in \cite[Section 10.2]{GMWZ1}; though, our purpose is
slightly different and we have to treat the degeneracy of the
viscosity matrix $B^{11}$ as comparing to the identity matrix in
\cite{GMWZ1}. Specifically, we prove the following:

\begin{lemma}\label{lem-auxi} Let $V = (V_1,V_2)\in \CC^{n-r}\times \CC^r$ be a solution of \eqref{auxi-eq}. We prove that there exists a constant
$C>0$ such that
\begin{equation}\label{auxi-Vbound}\begin{aligned}&|V|_{L^p} \le C
(|f|_{L^1}+|f|_{L^\infty})\\&|V_x|_{L^p} \le C
(|f|_{L^1}+|f|_{L^\infty} +
|f_x|_{L^p})\end{aligned}\end{equation}for any $1\le p\le \infty$.
\end{lemma}

\begin{proof}%[Proof of Lemma \ref{lem-auxi}]
We first integrate the equation \eqref{auxi-eq},
yielding
\begin{equation}\label{auxi-eqi} B^{11}V_{x}-A^1V =
f.\end{equation}

Consider the double $2n\times 2n$ boundary problem on $x\ge 0$
equivalent to \eqref{auxi-eqi}
\begin{equation} \label{auxi-eqbc}\begin{aligned} &\cB W_{x}-\cA W =
F\\& \Gamma W = 0\qquad \mbox{on}\qquad
\{x=0\}\end{aligned}\end{equation} where, defining $\phi_\pm(x) =
\phi(\pm x)$ for any function $\phi$ defined on $\RR$,
\begin{equation}\label{auxi-coeff}\begin{aligned} &W(x) = \begin{pmatrix}V_+(x)
\\V_-(x)\end{pmatrix},\\& \cA(x) = \begin{pmatrix}A^1_+(x)
&0\\0& -A_-^1(x)\end{pmatrix},\\& \cB(x) =
\begin{pmatrix}B^{11}_+(x) &0\\0& -B_-^{11}(x)\end{pmatrix},% =
\\& F(x) =
\begin{pmatrix}f_+(x)
\\-f_-(x)\end{pmatrix},\\&\Gamma W = V_+ -
V_-.\end{aligned}\end{equation} In what follows, we shall keep track
of variables $W$ as $(W_1^+,W_2^+,W_1^-,W_2^-)\in \CC^{n-r}\times
\CC^{r}\times \CC^{n-r}\times \CC^{r}$ in the obvious way
corresponding to matrix blocks as above. Notice also that $\cB$ is
degenerate in $W_1^\pm$-blocks.

Let $\cE_-(0)$ be the space of boundary values of decaying solutions
of \eqref{auxi-eqbc} when $F =0$. Then, we have $$\dim \cE_-(0) =
i$$ where $i$ is defined at the beginning of Section
\ref{shockprofile}. On the other hand, $\ker \Gamma$ has dimension
$n$. Thus, Assumption (D) then implies that $\ker \Gamma $ and
$\cE_-(0)$ have an $l=i-n$ dimensional intersection spanned by
$$(\phi_1(0),\phi_1(0)),\cdots,(\phi_l(0),\phi_l(0))$$ where
functions $\phi_i(0)$ are defined as in the paragraph just below
\eqref{Evans1}.

We define an augmented boundary condition $\tilde \Gamma$ with
property that \begin{equation}\label{bd-proj}\CC^{2n} = \ker \tilde
\Gamma \oplus \cE_-(0).\end{equation}

Since $\phi_j$ form a basis of the tangential space of the smooth
manifold $\{\bU^\delta(\cdot)\}$ defined as in (H3), without loss of
generality, we assume that the $j^{th}$ component of $\phi_j$ is not
zero. Thus, let us define
\begin{equation} \tilde \Gamma W = (W_1,\cdots,W_l,V_+ -
V_-)\end{equation} where $W_1,\cdots,W_l$ are the first $l$
components of $W\in \CC^{2n}$.

Now we consider the system
\begin{equation} \label{auxi-eqagbc}\begin{aligned} &\cB W_{x}-\cA W =
F\\& \tilde\Gamma W = 0\qquad \mbox{on}\qquad
\{x=0\}.\end{aligned}\end{equation} Since any solution of
\eqref{auxi-eqagbc} is also a solution of \eqref{auxi-eqbc}, we only
need to give an estimate for solutions of \eqref{auxi-eqagbc}. By
using the MZ conjugation \cite{MeZ1}, there is a uniformly bounded
transformation $\cC$ such that by setting $W = \cC Z$,
\eqref{auxi-eqagbc} gives
\begin{equation} \label{auxi-eqagbcZ}\begin{aligned} &\cB(\infty) Z_{x}-\cA(\infty) Z =
\bar F\\& \bar\Gamma Z = 0\qquad \mbox{on}\qquad
\{x=0\}\end{aligned}\end{equation} where $\bar \Gamma = \tilde
\Gamma \cC$. %Indeed, we first write \eqref{auxi-eqagbc} in form of
%$\tilde W_x - \tilde \cA W = \tilde F$, use the MZ conjugation, and
%then rewrite it as in \eqref{auxi-eqagbcZ}. We omit these steps.
%However,
Now let us define new variable $Y$ as
$$Y: = \cQ Z, \qquad \mbox{with}\qquad \cQ:=\begin{pmatrix}\begin{pmatrix}I_{n-r}&0\\b_1^{11}&b_2^{11}\end{pmatrix}
&0\\0&
\begin{pmatrix}I_{n-1}&0\\-b_1^{11}&-b_2^{11}\end{pmatrix}\end{pmatrix}.$$

Then $Y$ solves \begin{equation} \label{auxi-eqagbcY}\begin{aligned}
&\tilde Y_{x}-\cA(\infty)\cQ^{-1} Y = \hat F\\&Y =
(Y_1^+,Y_2^+,Y_1^-,Y_2^-)\\&\tilde Y = (0,Y_2^+,0,Y_2^-)\\&
\hat\Gamma Y = 0\qquad \mbox{on}\qquad
\{x=0\}\end{aligned}\end{equation} where $ \hat \Gamma = \tilde
\Gamma \cC\cQ^{-1}$ and $\hat F = \bar F \cQ^{-1}$.

Now by view of \eqref{auxi-coeff}, \eqref{ellcond}, and (H2),
eigenvalues $\mu_j$ of each block of $\cA(\infty)\cQ^{-1}$ are
distinct and nonzero. Thus, by performing a further transformation
if necessary, we could assume that $\cA(\infty)\cQ^{-1}$ is
diagonal. In these diagonalized coordinates, the system
\eqref{auxi-eqagbcY} consists of $2n$ ``uncoupled'' equations:
$$\begin{aligned}
&-\mu_{j\pm} Y_{1\pm} = \hat F_{1\pm}\\
(Y_{2\pm})_x &- \mu_{j\pm} Y_{2\pm} = \hat F_{2\pm}
\end{aligned}$$
where note that $Y_{i\pm}$ are the projections of $Y_i$ on the
growing (resp. decaying) eigenspaces of $\cA$ associated to
eigenvalues $\mu_{j\pm}$. In particular, $\pm \mu_{j\pm} >0$.

From equations for $Y_1$, it is clear that
\begin{equation}\label{auxi-estY1}|Y_{1}|_{L^p}\lesssim |\tilde
F_1|_{L^p}\qquad \forall p\ge 1.\end{equation}

Meanwhile, $Y_{2\pm}$ satisfies $$\begin{aligned} Y_{2+}(x) &=
\int_x^\infty e^{\mu_{j+}(x-y)} \hat F_{2+}(y)~dy,\\
Y_{2-}(x) &= e^{\mu_{j-}x}Y_{2-}(0)+\int_0^x e^{\mu_{j-}(x-y)} \hat
F_{2-}(y)~dy.\end{aligned}$$

Thus, this yields \begin{equation}\label{auxi-estY2}
\begin{aligned}&|Y_{2+}|_{L^p} \lesssim |\hat F_{2+}|_{L^p},\\&
|Y_{2-}|_{L^p} \lesssim |\hat F_{2-}|_{L^p} + |Y_{2-}(0)|,
\end{aligned}\qquad \forall p\ge 1.\end{equation}

Now since $Y_{2-}(0)\in \cE_-(0)$, by view of \eqref{bd-proj} as
bounded projections and the fact that all our transformations and
their inverses are bounded, we must have
$$|Y_{2-}(0)| \lesssim |\hat \Gamma Y_{-}(0)| \lesssim |\hat \Gamma Y(0)| +
|Y_+(0)| \lesssim |Y_+|_{L^\infty}.$$

Altogether,  we obtain
\begin{equation}\label{auxi-estY}\begin{aligned} &|Y|_{L^p}\lesssim
|\tilde F|_{L^1}+ |\tilde F|_{L^\infty}, \qquad \forall p\ge 1
\end{aligned}
\end{equation} which proves the first bound in \eqref{auxi-Vbound}. Estimates for derivatives are then
immediate by differentiating the equations of $Y_{1}$ in
\eqref{auxi-eqagbcY} and by solving equations \eqref{auxi-eqagbcY}
for $Y_{2x}$ in terms of $Y$ and $\hat F$. Thus, we have proved the
lemma as claimed.
\end{proof}

\section{Independence of the pointwise Green bounds}\label{indep-G}
In this section we comment on independence of the pointwise Green
function estimates. The high--frequency estimate \eqref{boundcS2}
can be derived entirely from auxiliary nonlinear energy estimates as
done in \cite{Z4}; see also Proposition 3.6, \cite{NZ2}, for a great
simplification. Whereas, the independency for the low--frequency
estimate \eqref{boundcS1} can be seen by first proving the following
slightly--weaker version of Proposition \ref{prop-resLF},
independent of the pointwise bounds
\eqref{ptbounds},\eqref{ptbounds-der}. A similar version can be done
for Proposition \ref{prop-resLFH5}.

\begin{proposition}[Low-frequency bounds]\label{prop-indepG} Under the hypotheses of Theorem \ref{theo-nonlin}, for $\lambda \in \Gamma^{\tilde \xi}$ and $\rho :=|(\tilde
\xi,\lambda)|$, $\theta_1$ sufficiently small, there holds the
resolvent bound \begin{equation}\label{bound-indepG} |(L_{\tilde
\xi}-\lambda)^{-1}\partial_{x_1}^\beta f|_{L^p(x_1)} \le
C\rho^{-3/2+(1-\alpha)\beta}(|
f|_{L^1(x_1)}+|\partial_{x_1}f|_{L^1(x_1)}),\end{equation} for all
$2\le p\le \infty$, $\beta =0,1$, and $\alpha$ defined as in
\eqref{alpha}.
\end{proposition}

\begin{proof} Certainly by Proposition \ref{prop-resLF}, we only need to prove the bound in the case
$\beta=1$. In the case of undercompressive shocks, the bound is
clear by applying \eqref{bound-indepG} with $\beta=0,\alpha=1$ and
$f$ replaced by $\partial_{x_1}f$:
\begin{equation}|(L_{\tilde
\xi}-\lambda)^{-1}\partial_{x_1}f|_{L^p(x_1)} \le
C\rho^{-3/2}|\partial_{x_1}f|_{L^1(x_1)}.\end{equation}

Now for the case of Lax or overcompressive shocks, we use the
Kreiss--Kreiss trick as in the proof of Proposition
\ref{prop-resLF}, that is, write $U = V + U_1$ where $V$ solves the
auxiliary problem \eqref{auxi-eq}. From the estimate
\eqref{auxi-Vbound} of $V$ and the inequality: $|f|_{L^\infty}\le
|\partial_{x_1}f|_{L^1(x_1)}$ (with $f(+\infty)=0$), we have
\begin{equation}\label{used-vbound}\begin{aligned}&|V|_{L^p} +|V_x|_{L^1}\le C
(|f|_{L^1(x_1)}+|\partial_{x_1}f|_{L^1(x_1)}), \quad \forall p\ge
1.\end{aligned}\end{equation} Thus, replacing \eqref{Vbound} by this
inequality and following the proof of Proposition \ref{prop-resLF},
we obtain the bound \eqref{bound-indepG}.
\end{proof}

\begin{proof}[Proof of Theorem \ref{theo-nonlin}, provided \eqref{bound-indepG}.]
The resolvent estimate \eqref{bound-indepG} is only weaker than
\eqref{res-bound} by %requiring
a stronger norm on $f$. We thus can certainly follow the proof in
Section \ref{sec-estS1}, yielding a low--frequency estimate like
\eqref{boundcS1}, but again weaker by a stronger norm on $f$,
namely, $|f|_{L^1(x)}+|\partial_{x_1}f|_{L^1(x)}$. With this
slightly weaker estimate for $\cS_1$, we can follow word by word the
proof of the theorem in Section \ref{sec-stab}, noting that the
higher derivatives of $f$ (then of $U$) can then be estimated by the
energy estimate \eqref{Hs}. Thus, we obtain the theorem without
requiring any further regularity on the structures of the system.
\end{proof}

%\begin{remark} \textup{As also mentioned in \cite[Remark 10.1]{GMWZ1},
%the auxiliary construction fails in the case of undercompressive
%shocks. Thus, our proof for this case still depends on the pointwise
%bounds as in \eqref{ptbounds}. It could be interesting for further
%investigation to provide an alternative proof similarly to the above
%for this shock case in which the pointwise Green function bounds are
%not used.}
%\end{remark}

\end{document}